\documentclass[
final
]{dmtcs-episciences}

\usepackage[utf8]{inputenc}
\usepackage{subfigure}

\usepackage{amsfonts,amsthm,amsmath,amssymb}

\usepackage{pict2e}

\usepackage{hyperref,url}

\usepackage{tkz-graph}
\usepackage{pgf,tikz}
\usetikzlibrary{arrows}

\theoremstyle{plain}
\newtheorem{theorem}{Theorem}[section]
\newtheorem{lemma}[theorem]{Lemma}
\newtheorem{corollary}[theorem]{Corollary}

\newtheorem{proposition}[theorem]{Proposition}

\newtheorem{Observation}[theorem]{Observation}

\theoremstyle{definition}

\newtheorem{conjecture}[theorem]{Conjecture}
\newtheorem{remark}[theorem]{Remark}
\newtheorem{?}[theorem]{Problem}

\def\blue{\textcolor{blue}}
\def\red{\textcolor{red}}

\newcommand{\N}{\mathbb{N}}

\def\S{\mathfrak{S}}
\def\D{\mathfrak{D}}
\def\T{\mathfrak{T}}
\def\C{\mathfrak{C}}

\def\A{\mathcal{A}}
\def\B{\mathcal{B}}

\def\last{\mathrm{last}}
\def\deg{\mathrm{deg}}
\def\capp{\mathrm{cap}}

\def\zero{\mathrm{zero}}

\def\asc{\mathrm{asc}}

\def\dist{\operatorname{dist}}

\def\max{\operatorname{max}}

\def\I{{\bf I}}

\def\rmin{\operatorname{rmin}}
\def\satu{\operatorname{satu}}

\def\rep{\operatorname{rep}}

\author{Chunyan Yan\affiliationmark{1}
  \and Zhicong Lin\affiliationmark{2}\thanks{Corresponding author. E-mail address: linz@sdu.edu.cn.}
  \thanks{Supported in part by the National Science Foundation of China grant 11871247 and the project of Qilu Young Scholars of Shandong University.}
  }
\title[Inversion sequences avoiding pairs of patterns]{Inversion sequences avoiding pairs of patterns}
\affiliation{
 School of Science, Jimei University, Xiamen 361021, P.R. China\\
  Research Center for Mathematics and Interdisciplinary Sciences, Shandong University, Qingdao 266237, P.R. China}
\received{2019-12-11}
\revised{2020-6-3}
\accepted{2020-6-4}
\begin{document}
\publicationdetails{22}{2020}{1}{23}{5964}
\maketitle
\begin{abstract}
 The enumeration of inversion sequences avoiding a single pattern was initiated by Corteel, Martinez, Savage and Weselcouch, and Mansour and Shattuck, independently. Their work has sparked various  investigations of generalized patterns in inversion sequences, including  patterns of relation triples  by Martinez and Savage, consecutive patterns by Auli and Elizalde, and vincular patterns by Lin and Yan. 
In this paper, we carried out the systematic study of inversion sequences avoiding two patterns of length $3$. Our enumerative results establish further connections to the OEIS sequences and some classical combinatorial objects, such as restricted permutations, weighted ordered trees and set partitions.
Since patterns of relation triples are some special multiple patterns of length $3$, our results complement the work by Martinez and Savage. In particular, one of their conjectures regarding the enumeration of $(021,120)$-avoiding inversion sequences is solved. 
\end{abstract}

\section{Introduction}
A {\em permutation} of $[n]:=\{1,2,\ldots,n\}$ is a word $\pi_1\pi_2\cdots\pi_n$ such  that 
$\{\pi_1,\pi_2,\ldots,\pi_n\}=[n]$. An {\em inversion sequence} of length $n$ is an integer sequence $(e_1,e_2,\ldots,e_n)$ with the restriction $0\leq e_i\leq i-1$ for all $1\leq i\leq n$. There are several interesting bijections between $\S_n$, the set of permutations of $[n]$, and $\I_n$, the set of inversion sequences of length $n$, which are known as {\em codings} of permutations (see~\cite{cor,kl,ms,sav} and the references therein). Perhaps, the most natural coding in the literature is the so-called {\em Lehmer code} $\Theta: \S_n\rightarrow\I_n$ of permutations define 
for $\pi=\pi_1\pi_2\cdots\pi_n\in\S_n$ as 
$$
\Theta(\pi)=(e_1,e_2,\ldots,e_n),\quad\text{ where $e_i:=|\{j: j<i\text{ and }\pi_j>\pi_i\}|$}.
$$
Permutations and inversion sequences are viewed  as words over $\N$. A word $W=w_1w_2\cdots w_n$ is said to avoid the word (or {\em pattern}) $P=p_1p_2\cdots p_k$ ($k\leq n$) if there does not exist $i_1<i_2<\cdots<i_k$ such that the subword $w_{i_1}w_{i_2}\cdots w_{i_k}$ of $W$ is order isomorphic  to $P$. For example, the word $W = 32421$ contains the pattern $231$, because the subword $w_2w_3w_5 = 241$ of $W$ has the same relative order as $231$. However, $W$ avoids both  $101$ and $123$.
For a set of words $\mathcal{W}$ and a sequence of patterns  $P_1,\ldots,P_r$, let us denote $\mathcal{W}(P_1,\ldots,P_r)$ the set of words in $\mathcal{W}$ which avoid all patterns $P_i$ for $i=1,\ldots,r$.

The enumerative aspect of pattern avoiding permutations has been the focus for over a half century since MacMahon's investigation  of word statistics and Knuth's consideration of stack-sorting algorithm (see Kitaev's monograph~\cite{ki}). Nevertheless, the systematic study of patterns in  inversion sequences was initiated only around 2015 by Corteel--Martinez--Savage--Weselcouch~\cite{cor} and Mansour--Shattuck~\cite{mash}, where the enumeration of inversion sequences avoiding a single pattern of length $3$ were nearly completed.  Two remarkable connections are 
$$
|\I_n(021)|=S_n\quad\text{and}\quad|\I_n(000)|=E_n,
$$
where $S_n$ is the $n$-th {\em large Schr\"oder number} and $E_n$ is the $n$-th {\em Euler up-down number}. Their work has inspired lots of further investigations~\cite{auli,auli2,bbgr,bgrr,cjl,kl,kl2,lin,lin2,ms,lyan,yan} of patterns in inversion sequences from the perspective of enumeration.
Especially, it has sparked various studies  of generalized patterns in inversion sequences, including  patterns of relation triples  by Martinez and Savage~\cite{ms}, consecutive patterns by Auli and Elizalde~\cite{auli,auli2}, and vincular patterns by Lin and Yan~\cite{lyan}. 

Motivated by the above works, the objective of this paper is to investigate systematically  
inversion sequences avoiding two patterns of length $3$. Our study complements the work by Martinez and Savage~\cite{ms} on enumeration of inversion sequences avoiding a fixed triple of binary relations $(\rho_1,\rho_2,\rho_3)\in\{<,\,>,\,\leq,\,\geq,\,=,\,\neq,\,-\}^3$. 
Following~\cite{ms}, consider
  the set $\I_n(\rho_1,\rho_2,\rho_3)$ consisting of those $e\in\I_n$  with no $i<j<k$ such that 
$$e_i\,\rho_1\,e_j,\quad e_j\,\rho_2\,e_k\quad\text{and}\quad
e_i\,\rho_3\,e_k.$$ Here the relation $''-''$ on a set $S$ is all of $S\times S$, i.e.,
$x\,''\!\!-''y$ for all $x,y\in S$. For example, we have   $\I_n(\geq,\neq,\geq)=\I_n(101,110,201,210)$, $\I_n(<,-,<)=\I_n(012,021,011)$ and $\I_n(>,-,\leq)=\I_n(101,102)$. In general,  patterns of relation triples are some special multiple patterns of length $3$, so our study complements the work~\cite{ms} in this sense.

\begin{table}
\hrule
\vspace{.1in}

{\small
\begin{tabbing}
xxxxxxxxxxxxxxxx\=xxxxxxxxxxxxxxxxxxxxxxxxxxx\= xxxxxxxxxxxxxx\=xxxxxxxxxxxxx\= \kill

Pattern pair $p$\> $a_n=|\I_n(p)|$ counted by: \> solved?\> OEIS\>  $a_8$, equiv class\\
\\

$(001,010)$\> 	 $n$ \> Sec.~\ref{sec:n}  \> A000027    \>8,A\\
$(001,011)$\> 	 $n$ \> Sec.~\ref{sec:n}  \> A000027    \>8,B\\
$(001,012)$\> 	 $n$ \> Sec.~\ref{sec:n}  \> A000027    \>8,C\\


$(001,110)$\> 	Lazy caterer sequence \> Sec.~\ref{sec:29} \> A000124    \>29,A\\
$(001,021)$\> 	 Lazy caterer sequence \> Sec.~\ref{sec:29} \> A000124    \>29,B\\
$(001,120)$\> 	 Lazy caterer sequence \> Sec.~\ref{sec:29} \> A000124    \>29,C\\
$(000,001)$\> 	 Fibonacci number $F_{n+1}$ \> Sec.~\ref{sec:fibo}  \> A000045    \>34\\
$(001,100)$\> 	 $F_{n+2}-1$ \> Sec.~\ref{sec:54}  \> A000071   \>54\\

$(001,210)$\> 	 Cake number ${n\choose 3}+n$ \> Sec.~\ref{sec:cake} \> A000125    \>64\\
$(000,011)$\> 	 $2^{n-1}$ \> Sec.~\ref{sec:power} \> A000079    \>128,A\\
$(001,101)$\> 	 $2^{n-1}$ \> Sec.~\ref{sec:power}  \> A000079    \>128,B\\
$(001,102)$\> 	 $2^{n-1}$ \> Sec.~\ref{sec:power} \> A000079    \>128,C\\
$(001,201)$\> 	 $2^{n-1}$ \> Sec.~\ref{sec:power}  \> A000079    \>128,D\\
$(010,012)$\> 	 $2^{n-1}$ \> Sec.~\ref{sec:power}  \> A000079    \>128,E\\
$(011,012)$\> 	 $2^{n-1}$ \> Sec.~\ref{sec:power} \> A000079    \>128,F\\

$(110,012)$\> 	 $2^n-n$ \> Sec.~\ref{sec:248}  \> A000325    \>248,A\\
$(012,021)$\> 	 $2^n-n$ \> Sec.~\ref{sec:248} \> A000325   \>248,B\\


$(012,201)$\> 	 $|\S_n(321,2143)|$ \>Sec.~\ref{sec:411} \> A088921    \>411,A\\
$(012,210)$\> 	 $|\S_n(321,2143)|$ \> Sec.~\ref{sec:411}\> A088921    \>411,B\\
$(011,102)$\> 	 $F_{2n-1}$ \> Sec.~\ref{sec:610}  \> A001519   \>610,A\\
$(012,102)$\> 	 $F_{2n-1}$ \> Sec.~\ref{sec:610}  \> A001519   \>610,B\\
$(012,120)$\> 	 $F_{2n-1}$ \> Sec.~\ref{sec:610}  \> A001519   \>610,C\\
$(010,011)$\> 	 $\sum_{k=0}^{n-1}(n-k)^k$ \> Sec.~\ref{sec:733} \> A026898    \>733\\
$(011,021)$\> 	 Catalan number $C_n$ \> Sec.~\ref{sec:cat} \> A000108    \>1430,A\\
$(010,021)$\> 	 Catalan number $C_n$ \> Sec.~\ref{sec:cat} \> A000108    \>1430,B\\
$(011,201)$\> 	 $|\I_n(-,>,\geq)|=|\I_n(\neq,\geq,\geq)|$ \> Open \> A279555    \>3091,A\\
$(011,210)$\> 	 $|\I_n(-,>,\geq)|=|\I_n(\neq,\geq,\geq)|$ \> Open \> A279555    \>3091,B\\
$(021,120)$\> 	 $1+\sum_{i=1}^{n-1}{2i\choose i-1}$ \>Sec.~\ref{sec:4.1}\> A279561    \>4082,A\\
$(102,120)$\> 	 $1+\sum_{i=1}^{n-1}{2i\choose i-1}$ \>Sec.~\ref{sec:4.2}\> A279561    \>4082,B\\
$(110,102)$\> 	 $1+\sum_{i=1}^{n-1}{2i\choose i-1}$ \>Sec.~\ref{sec:4.3}\> A279561    \>4082,C\\

$(010,100)$\> 	 Bell number $B_n$ \> Sec.~\ref{sec:bell}  \> A000110    \>4140,D\\
$(011,101)$\> 	 Bell number $B_n$ \> Sec.~\ref{sec:bell} \> A000110    \>4140,E\\
$(011,110)$\> 	 Bell number $B_n$ \> Sec.~\ref{sec:bell}  \> A000110   \>4140,F\\


$(101,021)$\> 	 $|\S_n(4123,4132,4213)|$ \>Sec.~\ref{sec:tree}\> A106228    \>5798,B\\

$(021,201)$\> 	 Large Schr\"oder number $S_n$ \>Sec.~\ref{sec:sch}\> A006318    \>8558,A\\
$(021,210)$\> 	 Large Schr\"oder number $S_n$ \>Sec.~\ref{sec:sch}\> A006318    \>8558,B\\
$(101,110)$\> 	 indecomposable set partitions \> Sec.~\ref{sec:irre} \> A074664    \>11624
\vspace{.1in}
\end{tabbing}
}

\hrule
\vspace{.1in}

\caption{Pattern pairs whose avoidance sequences appear to match sequences in the OEIS.}
\label{invseq:pairs}
\end{table}


\begin{table}
\hrule
\vspace{.1in}

{\small
\begin{tabbing}
xxxxxxxxxxxxxx\=xxxxxxxxxxxxxxxxxxxxxxxxxxx\= xxxxxxxxxxxxxxxxx\=xxxxxxxxxxx\= \kill

Pattern pair $p$\> $a_n=|\I_n(p)|$ counted by: \> calculated?\> in OEIS?\>  $a_8$, equiv class\\
\\

$(000,012)$\> $1, 2, 4, 5, 21, 0, 0,\ldots$	  \> Ultimately zero\> new    \>0\\

$(100,012)$\> $1, 2, 5, 12, 27, 56, 110,\ldots$	  \> Open\> New    \>207\\

$(101,012)$\> 	 $|\S_n(321,2413,3142)|$ \> In~\cite[Sec.~2.8]{ms}\> A034943    \>351\\
$(000,021)$\> $1, 2, 5, 14, 39, 111, 317,\ldots$	  \> Open\> New    \>911\\


$(000,102)$\> $1, 2, 5, 14, 40, 121, 373,\ldots$	  \> Open\> New    \>1181\\

$(000,010)$\> $1, 2, 4, 10, 29, 95, 345,\ldots$	  \> Open\> A279552    \>1376\\


$(011,120)$\> $1, 2, 5, 14, 42, 132, 431,\ldots$	  \> Open\> New    \>1452\\

$(100,011)$\> 	 Nexus numbers \> In~\cite[Sec.~2.15]{ms}  \> A047970   \>2048\\
 

$(021,102)$\> $1, 2, 6, 20, 66, 213, 683,\ldots$	  \> Open\> New    \> 2211\\

$(010,102)$\> $1, 2, 5, 15, 51, 186, 707,\ldots$	  \> Open\> New    \>2763\\

$(000,120)$\> $1, 2, 5, 15, 50, 185, 737, \ldots$	  \> Open\> New    \>3126\\

$(010,120)$\> $1, 2, 5, 15, 52, 201, 845, \ldots$	  \> Open\> A279559    \>3801\\

$(010,110)$\> $1, 2, 5, 15, 52, 201, 847,  \ldots$	  \> Open\> New    \> 3836\\


$(010,101)$\> 	 Bell number $B_n$ \> In~\cite[Sec.~2.17]{ms}  \> A000110   \>4140,A\\
$(000,101)$\> 	 Bell number $B_n$ \> In~\cite[Sec.~6]{cjl}  \> A000110    \>4140,B\\
$(000,110)$\> 	 Bell number $B_n$ \> In~\cite[Sec.~6]{cjl}  \> A000110   \>4140,C\\


$(100,021)$\> $1, 2, 6, 21, 78, 297, 1144,  \ldots$	  \> Open\> New    \> 4433,A\\
$(110,021)$\> Wilf-equiv to 4433,A (Sec.~\ref{sec:wilf})	  \> Open\> New    \> 4433,B\\

$(010,201)$\> $1, 2, 5, 15, 53, 214, 958, \ldots$	  \> Open\> New    \> 4650,A\\
$(010,210)$\> Wilf-equiv to 4650,A (Sec.~\ref{sec:wilf})	  \> Open\> New    \> 4650,B\\


$(100,102)$\> $1, 2, 6, 21, 80, 318, 1305,  \ldots$	  \> Open\> New    \> 5487\\


$(102,210)$\> $1, 2, 6, 22, 87, 351, 1416,\ldots$	  \> Open\> New    \> 5681\\


$(101,102)$\> 	 $|\S_n(4123,4132,4213)|$ \> In~\cite[Sec.~4]{cjl}\> A106228    \>5798,A\\


$(102,201)$\> $1, 2, 6, 22, 87, 354, 1465, \ldots$	  \> Open\> A279566    \> 6154\\


$(000,201)$\> $1, 2, 5, 16, 60, 257, 1218, \ldots$	  \> Open\> New    \> 6270,A\\
$(000,210)$\> Wilf-equiv to 6270,A (Sec.~\ref{sec:wilf})	  \> Open\> New    \> 6270,B\\
$(000,100)$\> $1, 2, 5, 16, 60, 260, 1267,\ldots$	  \> Open\> A279564    \> 6850\\

$(101,120)$\> $1, 2, 6, 22, 90, 397, 1859,\ldots$	  \> Open\> New    \> 9145\\

 $(100,120)$\> $1, 2, 6, 22, 92, 421, 2062,\ldots$	  \> Open\> New    \>  10646\\
$(110,120)$\> $1, 2, 6, 22, 92, 423, 2091,\ldots$	  \> Open\> A279570    \>  10950\\
$(100,110)$\> $1, 2, 6, 22, 93, 437, 2233, \ldots$	  \> Open\> New    \> 12227\\
$(100,101)$\> $1, 2, 6, 22, 93, 439, 2267,\ldots$	  \> Open\> New    \> 12628\\

$(120,201)$\> $1, 2, 6, 23, 102, 498, 2607,\ldots$	  \> Open\> New    \> 14386\\

$(120,210)$\> $1, 2, 6, 23, 102, 499, 2625, \ldots$	  \> Open\> A279573    \> 14601\\

$(110,201)$\> $1, 2, 6, 23, 103, 512, 2739,\ldots$	  \> Open\> New    \> 15464\\

$(101,210)$\> $1, 2, 6, 23, 103, 513, 2763,\ldots$	  \> Open\> New    \>  15816\\

$(101,201)$\> 	 $|\S_n(21\bar{3}54)|$ \> In~\cite[Sec.~2.27]{ms}\> A117106    \>17734,A\\
$(100,210)$\> 	 $|\S_n(21\bar{3}54)|$ \> In~\cite[Sec.~3.3]{bgrr}\> A117106    \>17734,B\\
$(100,201)$\> 	 $|\S_n(21\bar{3}54)|$ \> In~\cite[Sec.~2.27]{ms}\> A117106    \>17734,C\\
$(110,210)$\> 	 $|\S_n(21\bar{3}54)|$ \> In~\cite[Sec.~2.27]{ms}\> A117106    \>17734,D\\
$(201,210)$\> 	$|\S_n(MMP(0,2,0,2))|$  \> In~\cite[Sec.~2.29]{ms}\> A212198    \>23072


\vspace{.1in}

\end{tabbing}
}

\hrule
\vspace{.1in}

\caption{Pattern pairs whose avoidance sequences either have been studied in~\cite{bgrr,ms,cjl} or  appear to be new in the OEIS.}
\label{invseq:pairs2}
\end{table}


Our work is also parallel to the work by Baxter and Pudwell~\cite{bp} on pairs of patterns of length $3$ in {\em ascent sequences},
one of the most important subsets of inversion sequences introduced by Bousquet-M\'elou et al.~\cite{bcdk} to encode the unlabeled $(2$+$2)$-free posets. It should be noted that the enumeration of pattern avoiding ascent sequences was first carried out by Duncan and Steingr\'imsson in~\cite{ds} and followed by other researchers in~\cite{bp,pu,yan0}, even a few years earier than the study on inversion sequences, to our surprise.

The enumerative results obtained in this paper are summarized in Table~\ref{invseq:pairs}, which establish connections to the OEIS sequences and some classical combinatorial objects, such as restricted permutations, weighted ordered trees and irreducible  set partitions. In particular, a conjecture of Martinez and Savage~\cite{ms} regarding the enumeration of $(021,120)$-avoiding inversion sequences is solved. Our work together with~\cite{bgrr,ms,cjl}  completely classify all the Wilf-equivalences for inversion sequences avoiding pairs of length-$3$ patterns. Moreover, two related {\em unbalanced} Wilf-equivalence conjectures, presented as Conjectures~\ref{conj:wilf} and~\ref{conj:wilf2}, arise in this work.  

Since all the Wilf-equivalences for pattern pairs $p$ of length $3$ can be distinguished by the number $|\I_8(p)|$, we shall denote the class of $p$-avoiding inversion sequences by class  $|\I_8(p)|$. Within a Wilf class, equivalence classes are labeled A,B,C, etc.; see the  last column of Table~\ref{invseq:pairs} or~\ref{invseq:pairs2}. For example, class 128B represents the class of $(001,101)$-avoiding inversion sequences.

\subsection*{Statistics on inversion sequences}
For an inversion sequence $e=(e_1,e_2,\ldots,e_n)\in\I_n$, define the following five classical statistics:
\begin{itemize}
\item $\asc(e):=|\{i\in[n-1]: e_i<e_{i+1}\}|$, the number of {\em ascents} of $e$;
\item $\dist(e):=|\{e_1,e_2,\ldots,e_n\}\setminus\{0\}|$, the number of {\em distinct positive entries} of $e$;
\item $\rmin(e):=|\{i\in[n]: e_i<e_j \text{ for all $j>i$\,}\}|$, the number of {\em right-to-left minima} of $e$;
\item $\zero(e):=|\{i\in[n]: e_i=0\}|$, the number of {\em zero entries} in $e$;
\item $\satu(e):=|\{i\in[n]:e_i=i-1\}|$, the number of {\em saturated entries} of $e$.
\end{itemize}
Note that, over inversion sequences, the first two statistics are {\em Eulerian}, while the last three statistics are {\em Stirling} (See for example~\cite{kl,fjlyz}). These five statistics on restricted inversion sequences have been extensively studied in the literature (cf.~\cite{cjl,cor,ds,fjlyz,kl,kl2,lyan,ms}) and will play important roles in solving some enumerative problems in this paper. 

\subsection*{Notations}
Throughout this paper, we use the notation 
$$
\chi(\mathsf{S})=
\begin{cases}
1,\quad\text{if the statement $\mathsf{S}$ is true};\\
0,\quad\text{otherwise}.
\end{cases}
$$
We will also use the Kronecker delta function: $\delta_{i,j}$ equals $1$ if $i=j$; and $0$, otherwise.

\section{Some simple results}

This section contains some results that are obtained by elementary discussions. Some of the avoiding classes are even proved to be the same as these that have been studied in the literature. For clarity, all of them are listed below: 
\begin{align*}
&\I_n(001,110)=\I_n(\geq,\neq,-),\quad\I_n(000,001)=\I_n(=,\leq,-), \quad\I_n(001,100)=\I_n(\geq,\leq,\neq),\\
&\I_n(000,011)=\I_n(\leq,=,-),\quad \I_n(011,012)=\I_n(<,\leq,-), \quad\I_n(012,201)=\I_n(\neq,<,\neq),\\
&\I_n(001,101)=\I_n(001),\qquad\,\,\,\, \I_n(001,102)=\I_n(001),\qquad\,\,\,\,\,\I_n(001,201)=\I_n(001),\\
&\I_n(021,210)=\I_n(021),\qquad\,\,\,\, \I_n(021,201)=\I_n(021).
\end{align*}
The reason why they are equal will be given  throughout this section.
\subsection{Classes 8(A,B,C): $n$}
\label{sec:n}

\begin{theorem}
For any $n\geq1$ and a pattern pair $p\in\{(001,010),(001,011), (001,012)\}$, we have $|\I_n(p)|=n$.
\end{theorem}
This result is an immediate consequence of the following three simple observations.
\begin{Observation}
An inversion sequence $e\in\I_n$  is $(001,010)$-avoiding if and only if  for some $t\geq1$,
$$
e_1<e_2<\cdots < e_t=e_{t+1}=\cdots=e_n.
$$
\end{Observation}

\begin{Observation}
An inversion sequence $e\in\I_n$  is $(001,011)$-avoiding if and only if either $e=(0,0,\ldots,0)$ or for some $t\geq2$,
$$
e_1<e_2<\cdots < e_t>e_{t+1}=e_{t+2}=\cdots=e_n=0.
$$
\end{Observation}

\begin{Observation}
An inversion sequence $e\in\I_n$ is $(001,012)$-avoiding if and only if either $e=(0,0,\ldots,0)$ or for some $t\geq2$,
$$
e_1< e_2=e_3=\cdots =e_t> e_{t+1}=e_{t+2}=\cdots=e_n=0.
$$
\end{Observation}

\subsection{Classes 29(A,B,C): Lazy caterer sequence}
\label{sec:29}
The integer sequence
$$
\{n(n-1)/2+1\}_{n\geq1}=\{1, 2, 4, 7, 11, 16, 22, 29, 37, 46,\ldots\}
$$
appear as~\cite[A000124]{oeis} and is called the {\em Lazy caterer sequence} or the sequence of {\em central polygonal numbers}, which enumerates $(132,321)$-avoiding permutations~\cite{sim}. This sequence also counts three inequivalent classes $\I_n(<,\neq,-)$, $\I_n(<,-,<)$ and $\I_n(\geq,\neq,-)$ of relation triple avoiding inversion sequences~\cite[Sec.~2.4]{ms}.

\begin{theorem}
For any $n\geq1$ and a pattern pair $p\in\{(001,110),(001,021), (001,120)\}$, we have $|\I_n(p)|={n\choose 2}+1$, the $n$-th central polygonal number.
\end{theorem}

This result is clear from the three characterizations below. 

\begin{Observation}\label{obs:lazy}
An inversion sequence $e\in\I_n$  is $(001,110)$-avoiding if and only if  for some $t\geq1$, 
$$
e_1<e_2<\cdots<e_t\geq e_{t+1}=e_{t+2}=\cdots=e_n.
$$
\end{Observation}

\begin{remark}
Comparing Observation~\ref{obs:lazy} with~\cite[Observation~11]{ms} we see that the two sets $\I_n(001,110)$ and $\I_n(\geq,\neq,-)$ are the same. 
\end{remark}

\begin{Observation}
An inversion sequence $e\in\I_n$  is $(001,021)$-avoiding if and only if  for some $t\geq1$ and $s\geq t$,
$$
e_1<e_2<\cdots< e_t=e_{t+1}=\cdots=e_s>e_{s+1}=e_{s+2}=\cdots =e_n=0.
$$
\end{Observation}

\begin{Observation}
An inversion sequence $e\in\I_n$  is $(001,120)$-avoiding if and only if  for some $t\geq1$ and $s\geq t$,
$$
e_1<e_2<\cdots<e_t=e_{t+1}=\cdots =e_s>e_{s+1}=e_{s+2}=\cdots=e_n =e_{t-1}.
$$
\end{Observation}

\subsection{Class 34 and Fibonacci numbers}
\label{sec:fibo}

The {\em $n$-th Fibonacci number} $F_n$ can be defined by the recurrence $F_n=F_{n-1}+F_{n-2}$ for $n\geq2$ and the initial values $F_0=0$ and $F_1=1$. 

\begin{theorem}
For $n\geq 1$, we have $\I_n(000,001)=\I_n(=,\leq,-)$. Consequently,  the cardinality of $\I_n(000,001)$ is $F_{n+1}$. 
\end{theorem}
\begin{proof}
It is obvious that $\I_n(000,001)=\I_n(=,\leq,-)$. The second statement then follows from~\cite[Theorem~4]{ms}.
\end{proof}

\subsection{Class 54: $F_{n+2}-1$}
\label{sec:54}

\begin{theorem}
For $n\geq 1$, we have $\I_n(001,100)=\I_n(\geq,\leq,\neq)$. Consequently,  the cardinality of $\I_n(001,100)$ is $F_{n+2}-1$. 
\end{theorem}

\begin{proof}
Observe that $e\in\I_n$ is $(001,100)$-avoiding if and only if for some $t,s$ with $1\leq t\leq s\leq n$,
$$
e_1<e_2<\cdots <e_t=e_{t+1}=\cdots =e_s>e_{s+1}>e_{s+2}>\cdots >e_n.
$$
Comparing this with the characterization of $(\geq,\leq,\neq)$-avoiding inversion sequences in the proof of~\cite[Theorem~13]{ms}, we have $\I_n(001,100)=\I_n(\geq,\leq,\neq)$ and thus $|\I_n(001,100)|=F_{n+2}-1$.
\end{proof}

\subsection{Class 64 and Cake numbers}
\label{sec:cake}
The $n$-th {\em Cake number}, ${n\choose 3}+n$, is the maximal number of pieces resulting from $n-1$ planar cuts through a cake. The sequence of Cake numbers is registered as~\cite[A000125]{oeis} and has a number of geometric or combinatorial interpretations. We have the following  interpretation of Cake numbers as $(001,210)$-avoiding inversion sequences. 

\begin{theorem}
For any $n\geq1$, we have $|\I_n(001,210)|={n\choose 3}+n$.
\end{theorem}

\begin{proof}
Observe that $e\in\I_n$ is  $(001,210)$-avoiding if and only if for some $t\geq1$ and $t\leq s\leq n$,
$$
e_1<e_2<\cdots<e_t=e_{t+1}=\cdots =e_s>e_{s+1}=e_{s+2}=\cdots=e_n.
$$
It follows that for a fixed $t$, $1\leq t\leq n$, there are $(n-t)(t-1)+1$ many inversion sequences with the above form. Summing over  $t$ from $1$ to $n$ gives the desired result.
\end{proof}

\subsection{Classes 248(A,B,C,D,E,F): Powers of $2$}
\label{sec:power}
Simion and Schmidt~\cite{sim} showed that $|\S_n(132,231)|=2^{n-1}$. Corteel et al.~\cite{cor} showed that the coding $\Theta: \S_n\rightarrow \I_n$ restricts to a bijection from $\S_n(132,231)$ to $\I_n(001)$ and thus $|\I_n(001)|$ has cardinality $2^{n-1}$. Martinez and Savage~\cite[Section~2.6]{ms} proved that the three classes $\I_n(<,\leq, -)$, $\I_n(<,\geq,-)$ and $\I_n(\leq,=,-)$ also have cardinality $2^{n-1}$. In this section, we prove more interpretations for $2^{n-1}$.

\begin{theorem}
For any $n\geq1$ and a pattern pair 
$$p\in\{(000,011),(011,012),(001,101), (001,102),(001,201),(010,012)\},$$ we have $|\I_n(p)|=2^{n-1}$.
\end{theorem}

\begin{proof}
Since $\I_n(000,011)=\I_n(\leq,=,-)$ and $\I_n(011,012)=\I_n(<,\leq,-)$, the result for $p$ equals $(000,011)$ or $(011,012)$ was proved in~\cite[Section~2.6]{ms}. 

It has already been characterized in~\cite{cor} that $e\in\I_n(001)$ if and only if for some $t\in[n]$,
$$
e_1<e_2<\cdots<e_t\geq e_{t+1}\geq e_{t+2}\geq\cdots\geq e_n.
$$
Thus, an inversion sequence $e$ avoids $001$ must necessarily avoid $101$, $102$ and $201$.  It follows that  $\I_n(p)=\I_n(001)$ for any $p\in\{(001,101), (001,102),(001,201)\}$  and the result for these three cases are true. 

For $p=(010,012)$, we observe that inversion sequences $e\neq(0,0,\ldots,0)$ in $\I_n(p)$ are those satisfying 
$$
e_1=e_2=\cdots=e_{t}< e_{t+1}\geq e_{t+2}\geq \cdots\geq e_n>0
$$
for some $t\in[n-1]$. On the other hand,  $(011,012)$-avoiding inversion sequences are those whose positive entries are decreasing. For any $e\in\I_n(011,012)$ with $e_t$ as the leftmost positive entry (if any), define $e'=(0,\ldots,0,e_t',\ldots,e_n')$ where for $t\leq i\leq n$, $e_i'$ equals the nearest positive entry to the left of $e_i$ if $e_i=0$; otherwise $e_i'=e_i$.  For instance, if $e=(0,0,0,3,0,2,0,0,1)\in\I_9(011,012)$, then $e'=(0,0,0,3,3,2,2,2,1)$. It is routine to check that the mapping $e\mapsto e'$ sets up a bijection between $\I_n(011,012)$ and $\I_n(010,012)$. This completes the poof of the theorem. 
\end{proof}

\subsection{Classes 248(A,B): $2^n-n$}
 \label{sec:248}
 Permutations with at most one descent are known as {\em Grassmannian permutations} after Lascoux and Sch\"utzenberger~\cite{ls}. Grassmannian permutations of length $n$ are enumerated by $2^n-n$, which is also the cardinality of three equivalence classes of  relation triples for inversion sequences as shown in~\cite{ms}. We relate $2^n-n$  to two pairs of patterns in inversion sequences. 
 
 \begin{theorem}
For any $n\geq1$ and a pattern pair $p\in\{(012,021),(110,012)\}$, we have $|\I_n(p)|=2^n-n$.
\end{theorem}
\begin{proof}
{\bf The case of $p=(012,021)$}: Observe that $e\in \I_n(012,021)$ are those satisfying for some $t\in[n]$, 
$$
0=e_1=e_2=\cdots e_{t-1}<e_t\text{ and  either $e_j=0$ or $e_j=e_t$ when $t+1\leq j\leq n$}.
$$
The number of inversion sequences satisfying the above condition is $(t-1)2^{n-t}$ for $t\geq2$. Thus, 
$$
|\I_n(p)|=1+\sum\limits_{t=2}^n(t-1)2^{n-t}=1+2^{n-2}\sum_{k=1}^{n-1}k\biggl(\frac{1}{2}\biggr)^{k-1}=2^n-n,
$$
where the last equality follows from 
\begin{equation*}
\sum_{k=1}^{n-1}kx^{k-1}=\biggl(\sum_{k=1}^{n-1}x^k\biggr)'=\frac{nx^{n-1}(x-1)+1-x^n}{(1-x)^2}.
\end{equation*}

{\bf The case of $p=(110,012)$}: It is sufficient to show that for $n\geq2$,
$$
a_n-a_{n-1}=2^{n-1}-1,
$$
where $a_n=|\I_n(p)|$. Note that $\I_n(110,012)$ consists of inversion sequences $e\in\I_n$ whose  positive entries are weakly decreasing and whenever $e_i=e_j=\ell>0$ for some $i<j$, then $e_{j+1}=e_{j+2}=\cdots=e_n=\ell$. For $e\in\I_{n-1}(110,012)$, let $e'=(e_1,e_2,\ldots,e_{n-1},a)\in\I_{n}(110,012)$, where $a$ equals the least positive entry  of $e$, if $e\neq(0,0,\ldots,0)$; and $0$, otherwise. It is clear that the mapping $e\mapsto e'$ is a one-to-one correspondence between $\I_{n-1}(110,012)$ and the set $\I_n(110,012)\setminus\widetilde{\I}_n(110,012)$, where $\widetilde{\I}_n(110,012)$  consists of inversion sequences $e\in\I_n\setminus\{(0,0,\ldots,0)\}$ whose  positive entries are  decreasing. So it remains to show that $|\widetilde{\I}_n(110,012)|=2^{n-1}-1$. Since for any fixed $t$ and $k$, $2\leq t\leq n$ and $1\leq k$, the number of inversion sequences $e\in\widetilde{\I}_n(110,012)$ with $\min\{i:e_i>0\}=t$ and $|\{i:e_i>0\}|=k$ is ${t-1\choose k}{n-t\choose k-1}$, we have 
\begin{align*}
|\widetilde{\I}_n(110,012)|=\sum_{t=2}^n\sum_{k\geq1}{t-1\choose k}{n-t\choose k-1}=\sum_{t=2}^n{n-1\choose t-2}=2^{n-1}-1,
\end{align*}
where the second equality follows from the Vandermonde identity. This completes the proof of the theorem. 
\end{proof}

\subsection{Classes 441(A,B): $(321,2143)$-avoiding permutations}
\label{sec:411}
 The $2143$-avoiding permutations are called  {\em vexillary permutations}~\cite{ls} and the $321$-avoiding vexillary permutations are related to the combinatorics of Schubert polynomials by the work of Billey, Jockush and Stanley~\cite{bjs}. It was shown that 
 $$
 |\S_n(321,2143)|=2^{n+1}-{n+1\choose 3}-2n-1,
 $$
 the sequence A088921 in the OEIS. In~\cite[Sec.~2.9]{ms}, Martinez and Savage showed that $(\neq,<,\neq)$-avoiding inversion sequences are counted by the same function as $321$-avoiding vexillary permutations. Regarding pairs of patterns in inversion sequences, there has the following enumerative result.
  
 \begin{theorem}
For any $n\geq1$ and a pattern pair $p\in\{(012,201),(012,210)\}$, we have $|\I_n(p)|=2^{n+1}-{n+1\choose 3}-2n-1$.
\end{theorem}
\begin{proof}
Since $\I_n(\neq,<,\neq)=\I_n(102,012,201)$ and every inversion sequence contains the pattern $102$ must also contain $012$, we have $\I_n(012,201)=\I_n(\neq,<,\neq)$ and so $|\I_n(012,201)|=2^{n+1}-{n+1\choose 3}-2n-1$. 

Observe that a sequence  $e\in\I_n$ is $(012,210)$-avoiding if and only if
\begin{itemize}
\item positive entries of $e$ are weakly decreasing and
\item $\dist(e)\leq 2$, and if $\dist(e)=2$ with $e_{\ell}$ as the leftmost least positive entry, then $e_{\ell}=e_{\ell+1}=\cdots=e_n$.  
\end{itemize}
Let $\I_{n,t}(012,210)$ be the set of all inversion sequences $e\in\I_n(012,210)$ with $e_t$ as the leftmost positive entry. 
Recall that $\dist(e)$ denotes the number of distinct positive entries of $e$. The cardinality of $\{e\in\I_{n,t}(012,210): \dist(e)=1\}$ is easily seen to be $(t-1)2^{n-t}$, while 
$$
|\{e\in\I_{n,t}(012,210): \dist(e)=2\}|=\sum_{\ell=1}^{n-t}{t-1\choose 2}2^{n-t-\ell}={t-1\choose 2}(2^{n-t}-1),
$$
where each term ${t-1\choose 2}2^{n-t-\ell}$ counts the number of $e\in\I_{n,t}(012,210)$ with $\dist(e)=2$ whose leftmost least positive entry is $e_{n-\ell-1}$. It then follows that 
$$
|\I_{n}(012,210)|=1+\sum_{t=2}^n\biggl((t-1)2^{n-t}+{t-1\choose 2}(2^{n-t}-1)\biggr),
$$
which can be simplified to $2^{n+1}-{n+1\choose 3}-2n-1$. This completes the proof of the theorem. 
\end{proof}

\subsection{Class 733: $\sum_{k=0}^{n-1}(n-k)^k$}
\label{sec:733}

\begin{theorem}
For $n\geq 1$, $|\I_n(010,011)|=\sum_{k=0}^{n-1}(n-k)^k$. 
\end{theorem}
\begin{proof}
Note that $\I_n(010,011)$ is the set of $e\in \I_n$ satisfying for some $t\in [n]$, 
$$0=e_1=e_2=\cdots=e_t<e_{t+1}\text{ and $e_i\neq e_j$ for $t\leq i<j\leq n$}.
$$ 
So for fixed $t\in[n]$, the number of inversion sequences satisfying the above condition is $t^{n-t}$. The result then follows. 
\end{proof}

\subsection{Classes 1430(A,B) and Catalan numbers}
\label{sec:cat}
Recall that a {\em Dyck path} of length  $n$ is a path in $\N^2$ from $(0,0)$ to $(n,n)$ using the east step $(1,0)$ and the north step $(0,1)$, which does not pass above the diagonal $y=x$. 
It is well known that the $n$-th {\em Catalan number} $C_n=\frac{1}{n+1}{2n\choose n}$ counts the length $n$ Dyck paths. One can represent a Dyck path of length $n$ by a weakly increasing inversion sequence $e$ with $e_i$ as the height of its $i$-th east step. 
 \begin{theorem}
For any $n\geq1$ and a pattern pair $p\in\{(011,021),(010,021)\}$, we have $|\I_n(p)|=C_n$.
\end{theorem}
\begin{proof}
Note that $\I_n(010,021)$ is exactly the set of weakly increasing inversion sequences of length $n$ and so $|\I_n(010,021)|=C_n$. The mapping $e\mapsto e'$, where $e'$ is constructed from $e$ by replacing each positive entry by $0$ except for all the first occurrences  of positive values, is a bijection between $\I_n(010,021)$ and $\I_n(011,021)$. The proof is complete. 
\end{proof}

\subsection{Classes 8558(A,B) and Schr\"oder numbers}
\label{sec:sch}
A {\em Schr\"oder $n$-path} is a path in $\N^2$ from $(0,0)$ to $(2n,0)$ using only the steps $(1,1)$ (up step), $(1,-1)$ (down step) and (2,0) (flat step).
The $n$-th Large Schr\"oder number $S_n$ counts the Schr\"oder $n$-paths, as well as the $(2413,4213)$-avoiding permutations known as {\em Separable permutations}. One of the most remarkable results (see~\cite{cor,ms,kl,kl2,mash}) in the enumeration of pattern avoiding inversion sequences is 
$$
|\I_n(021)|=S_n. 
$$
Since every inversion sequence contains a pattern $210$ or $201$ must also contain the pattern $021$, we have $\I_n(p)=\I_n(021)$ for $p\in\{(021,210),(021,201)\}$ and consequently  the following result holds. 
 \begin{theorem}
For any $n\geq1$ and a pattern pair $p\in\{(021,210),(021,201)\}$, we have $|\I_n(p)|=S_n$.
\end{theorem}

\section{Classes 610(A,B,C): Boolean permutations}
\label{sec:610}
 The bisection of the Fibonacci sequence 
 $$
 \{F_{2n-1}\}_{n\geq1}=\{1, 2, 5, 13, 34, 89, 233, 610, 1597,\ldots\}
 $$
 appears as A001519 in~\cite{oeis} and has many combinatorial interpretations, among which is the  number of $(321,3412)$-avoiding permutations known as {\em Boolean permutations}. Note that $a_n=F_{2n-1}$ satisfies the recurrence
 $$
 a_n=3a_{n-1}-a_{n-2}.
 $$
 Corteel et al.~\cite{cor} showed that $|\I_n(012)|$ shares the same recurrence relation above as $a_n$ and thus proving
 \begin{equation}\label{bis:fib}
 |\I_n(012)|=F_{2n-1}. 
 \end{equation}
 The following result provides more interpretations of $a_n$ in terms of pattern avoiding inversion sequences. 
  \begin{theorem}
For any $n\geq1$ and a pattern pair $p\in\{(012,102),(012,120),(011,102)\}$, we have $|\I_n(p)|=F_{2n-1}$.
\end{theorem}
 \begin{proof}
 Since an inversion sequence contains the pattern $102$ or $120$ must also contain the pattern $012$, we have $\I_n(p)=\I_n(012)$ for $p\in\{(012,102),(012,120)\}$. The result for these two cases then follows from~\eqref{bis:fib}. 
 
 Next, we deal with the case $p=(011,102)$. Consider the triangle  
 $$
 \A_{n,k}:=\{e\in\I_n(012):\last(e)=k\},
 $$
 where $\last(e)$ is the last entry of $e$. We claim that for $n\geq2$,
  \begin{equation}\label{eq:a0}
  |\A_{n,0}|=|\A_{n,1}|=a_{n-1}
  \end{equation} and 
  \begin{equation}\label{eq:ak}
 |\A_{n,k}|=|\A_{n-1,k-1}|\text{ for $2\leq k\leq n-1$}.
  \end{equation}
Notice that the mapping $\delta_k: (e_1,e_2,\ldots,e_{n-1})\mapsto (e_1,e_2,\ldots,e_{n-1},k)$ is a one-to-one correspondence between $\I_{n-1}(012)$ and $\A_{n,k}$ for $k=0,1$ and so Eq.~\eqref{eq:a0} follows. For fixed $k$, $2\leq k\leq n-1$, the mapping $(e_1,e_2,\ldots,e_{n-1})\mapsto(0,e_1',e_2',\ldots,e_{n-1}')$ with $e_i'=e_i+\chi(e_i>0)$ is a one-to-one correspondence between $\A_{n-1,k-1}$ and $\A_{n,k}$, and thus Eq.~\eqref{eq:ak} holds. 
For any $e\in\I_n$, we recall from the introduction that   $\satu(e)$ is the number of indices $i\in[n]$ such that $e_i=i-1$.
We aim to show that 
$$
\B_{n,k}:=\{e\in\I_n(011,102): \satu(e)=k+1\}
$$
has the same carnality as $\A_{n,k}$, which would finish the proof. We will do this by showing 
 \begin{equation}\label{eq:b0}
  |\B_{n,0}|=|\B_{n,1}|=|\I_{n-1}(011,102)|
  \end{equation} and 
  \begin{equation}\label{eq:bk}
 |\B_{n,k}|=|\B_{n-1,k-1}|\text{ for $2\leq k\leq n-1$}.
  \end{equation}
  Comparing~\eqref{eq:b0} and~\eqref{eq:bk} with~\eqref{eq:a0} and~\eqref{eq:ak} we have $|\B_{n,k}|=|\A_{n,k}|$ by induction.   
  
  It remains to prove~\eqref{eq:b0} and~\eqref{eq:bk}.  Clearly, the mapping 
  $(e_1,\ldots,e_{n-1})\mapsto (0,e_1,\ldots,e_{n-1})
  $ is a one-to-one correspondence between $\I_{n-1}(011,102)$ and $\B_{n,0}$ and so $|\I_{n-1}(011,102)|=|\B_{n,0}|$. For $e\in\B_{n,1}$, there exist only one index $j>1$ such that $e_j=j-1$ and we introduce the mapping $\rho: \B_{n,1}\rightarrow \I_{n-1}(011,102)$ by
  \begin{itemize}
  \item if $j-2$ does not appear as the entry of $e$, then let $\rho(e)=(e_2,e_3,\ldots,e_j-1,e_{j+1},\ldots,e_n)$;
  \item otherwise, let $\rho(e)$ be constructed by removing the entry $e_j$ from $e$ directly. 
  \end{itemize}
  For instance, the mapping reads  $\rho(0,0,0,3)=(0,0,2)$, $\rho(0,0,1,3)=(0,1,2)$, $\rho(0,0,2,0)=(0,1,0)$, $\rho(0,0,2,1)=(0,0,1)$ and $\rho(0,1,0,0)=(0,0,0)$. The crucial observation about $\rho$ is that $e_j$ is the largest entry of $e$ and   $\rho(e)$ is constructed by the second bullet if and only if $\rho(e)\in\B_{n-1,0}$. To see that $\rho$ is a bijection, we introduce its inverse $\rho^{-1}$ explicitly. For $e\in\I_{n-1}(011,102)$ with $e_j=\ell$ as the largest entry, define $\rho^{-1}: \I_{n-1}(011,102)\rightarrow \B_{n,1}$ as 
  \begin{itemize}
  \item if $e\notin\B_{n-1,0}$, then $\rho^{-1}(e)=(0,e_1,e_2,\ldots,e_j+1,e_{j+1},\ldots,e_n)$;
  \item otherwise, $e\in\B_{n-1,0}$ and let $\rho^{-1}(e)=(e_1,e_2,\ldots,e_{\ell+1},\ell+1, e_{\ell+2},\ldots,e_{n-1})$.
  \end{itemize}
  It is routine to check that $\rho^{-1}$ is the inverse of $\rho$, which proves $|\B_{n,1}|=|\I_{n-1}(011,102)|$ and~\eqref{eq:b0} is established. To see~\eqref{eq:bk}, for any $e\in\B_{n,k}$ ($2\leq k\leq n-1$) with $e_{\ell}$ as the largest entry, consider the mapping $e\mapsto e'$ where $e'$ is obtained from $e$ by removing $e_{\ell}$ directly. Since $e_{\ell}$ is the rightmost saturated entry of $e$, it is easy to  see that $e\mapsto e'$ sets up a one-to-one correspondence between $\B_{n,k}$ and $\B_{n-1,k-1}$, which completes the proof of the theorem. 
 \end{proof}

 \section{Classes 4082(A,B,C): $1+\sum_{i=1}^{n-1}{2i\choose i-1}$}
 \label{sec:4082}

In~\cite[Sec.~3.1.3]{ms}, Martinez and Savage showed that 
$$
|\I_n(>,\neq,-)|=1+\sum_{i=1}^{n-1}{2i\choose i-1}
$$
and conjectured the following Wilf-equivalence. 

\begin{conjecture}[Martinez and Savage~\cite{ms}]
\label{conj:ms}
For $n\geq1$, we have
$$
|\I_n(>,\neq,-)|=|\I_n(<,>,\neq)|.
$$
\end{conjecture}

The integer sequence $\{1+\sum_{i=1}^{n-1}{2i\choose i-1}\}_{n\geq1}$, whose generating function is  
\begin{equation}\label{gen:sava}
1+\sum_{n\geq1}\left(1+\sum_{i=1}^{n-1}{2i\choose i-1}\right)x^n=\frac{1-4x+\sqrt{-16x^3+20x^2-8x+1}}{2(x-1)(4x-1)},
\end{equation}
has been  registered as A279561 in the OEIS~\cite{oeis}. In this section, we will prove Conjecture~\ref{conj:ms} and two more interpretations for this integer sequence.

\subsection{Pattern pair $(021,120)$: proof of Conjecture~\ref{conj:ms}} 
\label{sec:4.1}
 Recall that a Dyck path $D$ can be represented as $D=h_1h_2\cdots h_n$, where $h_i$ is the height of its $i$-th east step satisfying $0\leq h_i< i$.  Let  $\D_n$ be the set of Dyck paths of  length $n$. Note that $|\D_n|=C_n$, the $n$-th Catalan number, whose generating function  is 
\begin{equation}\label{exp:C}
C(x):=1+\sum_{n\geq1}C_nx^n=\frac{1-\sqrt{1-4x}}{2x}. 
\end{equation}

For our purpose, we color each east step of a Dyck path  by black or red and call such a Dyck path a {\em colored Dyck path}. If the $i$-th east step has height $k$ and color red, then we write $h_i=\bar{k}$. It was observed in~\cite{cor} that an inversion sequence is $021$-avoiding if and only if its positive entries are weakly increasing. By this characterization, each sequence $e\in\I_n(021)$ can be  represented by a colored Dyck path $\mathcal{H}(e)=h_1h_2\cdots h_n$, called the {\em outline of $e$} in~\cite[Def.~2.1]{kl}, where 
$$
h_i=
\begin{cases}
e_i, \quad&\text{if $e_i>0$},\\
\bar{k}, &\text{if $e_i=0$ and $k=\max\{e_1,\ldots,e_i\}$}.
\end{cases}
$$ 
For example, the outline of $(0,1,0,0,2,0,4)$ is the colored Dyck path $\bar{0}1\bar{1}\bar{1}2\bar{2}4$ drawn in Fig.~\ref{fig:outline}. 
Let us denote by $\A_n$ the set of colored Dyck paths of length $n$ satisfying 
\begin{itemize}
\item[(a)] all the east step of height $0$ are colored red, and
\item[(b)] the first east step of each positive height is colored black.
\end{itemize}
\begin{figure}
\begin{center}
\begin{tikzpicture}[scale=.6]
\draw[step=1,color=gray](0,1) grid (7,8);
\draw[very thick](1,1)--(1,2) (1,2)--(2,2)(3,2)--(4,2) (4,2)--(4,3) (4,3)--(5,3) (6,3)--(6,5)(6,5)--(7,5) (7,5)--(7,8);
\draw[very thick,color=red](0,1)--(1,1)(2,2)--(4,2) (5,3)--(6,3);
\draw[color=blue](0,1)--(7,8);
\draw(0.5,0.5) node{$0$};
\draw(1.5,0.5) node{$1$};
\draw(2.5,0.5) node{$0$};
\draw(3.5,0.5) node{$0$};
\draw(4.5,0.5) node{$2$};
\draw(5.5,0.5) node{$0$};
\draw(6.5,0.5) node{$4$};
\end{tikzpicture}
\end{center}
\caption{The outline of an inversion sequence}
\label{fig:outline}
\end{figure}
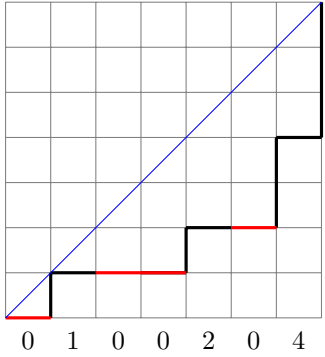
It is clear that the geometric representation $\mathcal{H}$ is a bijection between $\I_n(021)$ and $\A_n$. 
Let $\B_n$ be the set of colored Dyck paths in $\A_n$ such that only east steps with smallest positive height can have two colors. For example, the colored Dyck path in Fig.~\ref{fig:outline} is an element in $\A_7\setminus\B_7$. Note that $\I_n(<,>,\neq)=\I_n(021,120)$ and 
the geometric representation $\mathcal{H}$ induces a bijection between $\I_n(021,120)$ and $\B_n$, as any $(021,120)$-avoiding inversion sequence can not have zero entry after the first appearance of the entry with second smallest positive value. 
Let 
$$
S(x):=1+\sum_{n\geq1}|\B_n|x^n=1+x+ 2x^2+6x^3+21x^4+\cdots. 
$$
In  view of~\eqref{gen:sava}, Conjecture~\ref{conj:ms} is equivalent to the following result. 
\begin{theorem}\label{thm:sava}
We have 
$$
S(x)=\frac{1-4x+\sqrt{-16x^3+20x^2-8x+1}}{2(x-1)(4x-1)}.
$$
\end{theorem}

We decompose the proof of  the above theorem into the following two lemmas. Let $\C_n$ be the set of colored Dyck paths of length $n$ such that only east steps with height $0$, except the first step (which is colored black), can have two colors. Introduce the generating function
$$
H(x):=1+\sum_{n\geq0} |\C_n|x^n=1+x+3x^2+10x^3+35x^4+\cdots.
$$

\begin{lemma}\label{lem:S}
We have the following functional equation for $S(x)$:
\begin{equation}\label{eq:S}
S(x)=1+xH(x)/(1-x)+xC(x)(S(x)-1/(1-x)).
\end{equation}
\end{lemma}
\begin{proof}
For a colored Dyck path $D=h_1\cdots h_n\in\B_n$, let 
$$
k=\min\{i\in[n]: h_{i+1}=i\text{ or } i=n \}.
$$
Then, $D$ can be decomposed uniquely into a pair $(D_1,D_2)$ of Dyck paths, where 
$$
D_1:=h_2h_3\ldots h_k\quad\text{and}\quad D_2:=(h_{k+1}-k)(h_{k+2}-k)\cdots(h_n-k).
$$ We need to distinguish two cases:  
\begin{itemize}
\item if $D_1=\bar{0}\cdots\bar{0}$, the Dyck path consisting of $k-1$ red east steps of height $0$,  then $D_2\in\C_{n-k}$. 
\item otherwise   $D_1\neq\bar{0}\cdots\bar{0}$ and $D_1\in\B_{k-1}$, then $D_2\in\D_{n-k}$.
\end{itemize}
Turning this decomposition into generating functions then yields
$$
S(x)=1+x\biggl(\sum_{k\geq0}x^k\biggr)H(x)+x\biggl(S(x)-\sum_{k\geq0}x^k\biggr)C(x),
$$
which is equivalent to~\eqref{eq:S}. 
\end{proof}

\begin{lemma}
We have the following functional equation for $H(x)$:
\begin{equation}\label{eq:H}
H(x)=1+x(2H(x)-1)C(x).
\end{equation}
\begin{proof}
We will use similar decomposition  of the colored Dyck paths in $\C_n$ as in the proof of  Lemma~\ref{lem:S}. Let $\tilde\C_n$ be the set of colored Dyck paths of length $n$ such that only east steps of height $0$ can have two colors. Clearly, we have 
\begin{equation}
1+\sum_{n\geq1}|\tilde\C_n|x^n=2H(x)-1.
\end{equation}
For a colored Dyck path $D=h_1\cdots h_n\in\C_n$, let 
$$
k=\min\{i\in[n]: h_{i+1}=i\text{ or } i=n \}.
$$
Then, $D$ can be decomposed uniquely into a pair $(D_1,D_2)$ of Dyck paths, where 
$$
D_1:=h_2h_3\ldots h_k\in\tilde\C_{k-1}\quad\text{and}\quad D_2:=(h_{k+1}-k)(h_{k+2}-k)\cdots(h_n-k)\in\D_{n-k}.
$$
Turning this decomposition into generating function gives~\eqref{eq:H}.  
\end{proof}
\end{lemma}

\begin{proof}[{\bf Proof of Theorem~\ref{thm:sava}}]By~\eqref{eq:H}, we have 
\begin{equation}\label{exp:H}
H(x)=\frac{1-xC(x)}{1-2xC(x)}.
\end{equation}
Solving~\eqref{eq:S} for $S(x)$ gives 
\begin{equation}\label{exp:S}
S(x)=\frac{1+x(H(x)-C(x))/(1-x)}{1-xC(x)}.
\end{equation}
Substituting~\eqref{exp:C} and~\eqref{exp:H} into~\eqref{exp:S} then proves Theorem~\ref{thm:sava} after simplification. 
\end{proof}

\subsection{Pattern pair $(110,102)$}\label{sec:4.2}
For a Dyck path $D=h_1\cdots h_n\in\D_n$, let $\last(D)=h_n$ be the height of the last east step of $D$. Let us define $d_{n,m}$ the number of Dyck paths $D\in\D_n$ with $\last(D)=m$ and introduce $d_n(u):=\sum_{m=0}^{n-1}d_{n,m}u^m$. Let 
$$
D(u,x):=\sum_{n\geq1} d_n(u)x^n=x+(1+u)x^2+(2u^2+2u+1)x^3+\cdots
$$
be the enumerator of Dyck paths by the height of their last steps. It is known (cf.~\cite[Lem.~5.2]{cjl}) that $D(u,x)$ is algebraic and
\begin{equation}\label{alg:D}
D(u,x)=\frac{2x}{1-2x+\sqrt{1-4ux}}.
\end{equation}

For any $0\leq m<\ell\leq n$, denote by $\A_{n,m,\ell}$ the set of $(110,102)$-avoiding inversion sequences of length $n$ in which the largest entry is $m$ with the left-most occurrence  of $m$ in position $\ell$. Let $a_{n,m,\ell}=\#\A_{n,m,\ell}$. 
Then, the first few values of the arrays $[a_{n,m,\ell}]_{0\leq m<\ell\leq n}$ are:
\[ \left[\begin{array}{c}
1
\end{array}\right],
\left[\begin{array}{c c }
1 & 0 \\
0 & 1
\end{array}\right], \left[\begin{array}{c c c}
1 & 0 & 0\\
0 & 2 & 1\\
0 & 0 & 2 
\end{array}\right], \left[\begin{array}{c c c c}
1 & 0 & 0 & 0\\
0 & 3 & 2 & 1 \\
0 & 0 & 6 & 3 \\ 
0 & 0 & 0 & 5
\end{array}\right], \left[\begin{array}{c c c c c}
1 & 0 & 0 & 0 & 0\\
0 & 4 & 3 & 2 & 1 \\
0 & 0 & 12 & 8 & 4 \\ 
0 & 0 & 0 & 19 & 9\\
0 & 0 & 0 & 0 & 14
\end{array}\right].
\] 
We have the following recurrence relation for $a_{n,m,\ell}$.

\begin{lemma}For $1\leq m<\ell\leq n$, we have 
\begin{equation}\label{rec:110,102}
a_{n,m,\ell}=d_{n-1,m}-d_{\ell-1,m}+d_{\ell-1,m-1}-\chi(m=1,\ell=2)+\sum_{j=0}^m\sum_{k=\ell-1}^{n-1}a_{n-1,j,k},
\end{equation}
where $a_{n,0,\ell}=\delta_{\ell,1}$ for $1\leq \ell\leq n$.
\end{lemma}

\begin{proof}
For an element $e\in\A_{n,m,\ell}$, we distinguish two cases:
\begin{itemize}
\item[(a)] If there are at least two $m$ appearing as entries of $e$, then consider the sequence $\hat{e}$ obtained from $e$ by removing $e_{\ell}$. Since $e$ is $110$-avoiding, $e_n$ must equal $m$, which forces the entries of $\hat{e}$ to be  weakly increasing, for otherwise there will be a $102$ pattern occurs in $\hat{e}$ (this would contradict the fact that $\hat{e}$ is $102$-avoiding). Thus, $\hat{e}$ can be considered as a Dyck path whose last east step has height $m$ and the $(\ell-1)$-th east step has height smaller than $m$. There are exactly $d_{n-1,m}-d_{\ell-1,m}$ many members of $\A_{n,m,\ell}$ arise  in this case. 
\item[(b)] Otherwise, the letter $m$ occurs only once within $e$. Suppose that $j$ is the second largest  letter of $e$, then we need further to consider two cases:
\begin{itemize}
\item[(b1)] The left-most occurrence of $j$ within $e$ is in positions $1$ through $\ell-2$. Since $e$ is $102$-avoiding, each entry between the left-most $j$ and $m$ must  be a $j$. This forces all entries after $m$ must also equal $j$ because of $e$ is $110$-avoiding. Moreover, as $e$ is $102$-avoiding, the subsequence $(e_1,e_2,\ldots,e_{\ell-2})$ is weakly increasing and so members of $\A_{n,m,\ell}$ arising in this case  (with $j$ as second largest entry) are in bijection with Dyck paths of length $\ell-2$ with last east step of height $j$. Summing over all $0\leq j\leq m-1$, there are totally  $d_{\ell-1,m-1}-\chi(m=1,\ell=2)$ many members of $\A_{n,m,\ell}$ arise in this case. 
\item[(b2)]  The left-most occurrence of $j$ within $e$ is either in position $\ell-1$ or in positions $\ell+1$ through $n$. In either case, the entry $e_{\ell}=m$ is extraneous since the first $\ell-2$ entries of $e$ are governed by a $j$ either coming just before $m$ or at some position after $m$. The deletion of $e_{\ell}$ from $e$ then results in a member of $\A_{n,j,k}$ for some $0\leq j\leq m-1$ and $\ell-1\leq k\leq n-1$ and summing over these indices gives the last term in~\eqref{rec:110,102}. 
\end{itemize}
\end{itemize}
All the above cases together gives~\eqref{rec:110,102}, as desired. 
\end{proof}

For any $n\geq2$, let 
$$
a_{n}(u,v):=\sum_{\ell=2}^n\sum_{m=1}^{\ell-1}a_{n,m,\ell}u^mv^{\ell}. 
$$
Introduce the generating function $a(x;u,v)$ by
$$
a(x;u,v):=\sum_{n\geq2}a_n(u,v)x^n=uv^2x^2+(2uv^2+uv^3+2u^2v^3)x^3+\cdots.
$$
Next, we show that recurrence relation~\eqref{rec:110,102} can be translated into the following functional equation for $a(x;u,v)$. 

\begin{lemma}
The functional equation for $a(x;u,v)$ is 
\begin{multline}\label{fun:110,102}
a(x;u,v)=\frac{vx^2}{x-1}+\frac{vx}{1-v}D(uv,x)+\biggl(\frac{uvx-vx}{1-x}-\frac{v^2x}{1-v}\biggr)D(u,xv)+\\
+\frac{uv^2x}{(1-v)(1-uv)}a(x;uv,1)+\frac{u^2v^2x}{(1-u)(1-uv)}a(x;1,uv)-\frac{uv^2x}{(1-u)(1-v)}a(x;u,v).
\end{multline}
\end{lemma}
\begin{proof}
Multiplying both sides of~\eqref{rec:110,102} by $x^nu^mv^{\ell}$ and summing over all $1\leq m<\ell\leq n$ will yield a functional equation for $a(x;u,v)$. More precisely, we need to simplify the following terms coming from the right hand side of~\eqref{rec:110,102}:
\begin{align*}
&\quad\sum_{n\geq2}x^n\sum_{\ell=2}^nv^{\ell}\sum_{m=1}^{\ell-1}(d_{\ell-1,m-1}-d_{\ell-1,m})u^m\\
&=\sum_{n\geq2}x^n\sum_{\ell=2}^nv^{\ell}(ud_{\ell-1}(u)-d_{\ell-1}(u)+1)\\
&=(u-1)\sum_{\ell\geq2}d_{\ell-1}(u)v^{\ell}\sum_{n\geq\ell}x^n+\sum_{\ell\geq2}v^{\ell}\sum_{n\geq\ell}x^n\\
&=\frac{(u-1)vx}{1-x}D(u,vx)+\frac{v^2x^2}{(1-x)(1-vx)},
\end{align*}
\begin{align*}
&\quad\sum_{n\geq2}x^n\sum_{\ell=2}^nv^{\ell}\sum_{m=1}^{\ell-1}d_{n-1,m}u^m\\
&=\sum_{n\geq2}x^n\sum_{m=1}^{n-1}d_{n-1,m}u^m\sum_{\ell=m+1}^nv^{\ell}\\
&=\frac{v}{v-1}\sum_{n\geq2}\biggl(d_{n-1}(uv)-1-v^n(d_{n-1}(u)-1)\biggr)x^n\\
&=\frac{v}{v-1}\biggl(xD(uv,x)-\frac{x^2}{1-x}-vxD(u,vx)+\frac{v^2x^2}{1-vx}\biggr),
\end{align*}
\vskip0.1in
\begin{equation*}
-\sum_{n\geq2}x^n\sum_{\ell=2}^n\sum_{m=1}^{\ell-1}\chi(m=1,\ell=2)u^mv^{\ell}=-\sum_{n\geq2}uv^2x^n=\frac{uv^2x^2}{x-1},
\end{equation*}
and 
\begin{multline*}
\sum_{n\geq2}x^n\sum_{\ell=2}^nv^{\ell}\sum_{m=1}^{\ell-1}u^m\sum_{j=0}^m\sum_{k=\ell-1}^{n-1}a_{n-1,j,k}=\frac{uv^2x^2}{1-x}+\frac{uv^2x}{(1-v)(1-uv)}a(x;uv,1)\\
+\frac{u^2v^2x}{(1-u)(1-uv)}a(x;1,uv)-\frac{uv^2x}{(1-u)(1-v)}a(x;u,v),
\end{multline*}
where the last equality follows from the same manipulation as in~\cite[Lem.~3.5]{mash}.
Summing over all the above expressions gives the functional equation~\eqref{fun:110,102}.
\end{proof}

Finally, we will solve the functional equation~\eqref{fun:110,102} using the kernel method to obtain the following expression for $a(x;1,1)$.

\begin{theorem}\label{thm:110,102}
We have 
\begin{equation}\label{exp:110,102}
a(x;1,1)=\frac{x^2+x^2\sqrt{1-4x}}{(x-1)((3x-1)\sqrt{1-4x}-4x^2+5x-1)}.
\end{equation}
Equivalently, 
$$
|\I_n(110,102)|=1+\sum_{i=1}^{n-1}{2i\choose i-1}.
$$
\end{theorem}
\begin{proof}
We apply the kernel method and set the coefficients of $a(x;u,v)$ on both sides of~\eqref{fun:110,102} equal:
$$
1=-\frac{uv^2x}{(1-u)(1-v)}.
$$
Solving this equation for $u$ gives 
\begin{equation}\label{eq:u1}
u=u(x,v)=\frac{1-v}{1-v-v^2x}
\end{equation}
Let us introduce 
$$
\alpha=\alpha(x,v)=uv=\frac{v-v^2}{1-v-v^2x}.
$$
Substituting the expression for $u$ given by~\eqref{eq:u1} into~\eqref{fun:110,102},  and simplifying yields 
\begin{multline}\label{eq:alpha}
a(x;\alpha,1)=\frac{\alpha(v-1)}{v-\alpha}a(x;1,\alpha)+\frac{x(1-v)(1-\alpha)}{\alpha(1-x)}+\frac{\alpha-1}{\alpha}D(\alpha,x)\\+\frac{(1-\alpha)(v-\alpha-v\alpha-v^2x)}{v\alpha(1-x)}D(\alpha/v,vx). 
\end{multline}
Observe that $\alpha(x,w)=v$ whenever $w=w(x,v)$ is a root of the quadratic equation
\begin{equation}\label{eq:w1}
(1-vx)w^2-(1+v)w+v=0.
\end{equation}
Replacing  $v$ by $w$ in~\eqref{eq:alpha} gives
\begin{multline}\label{eq:alp2}
a(x;v,1)=E_1(v,w)a(x;1,v)+E_2(v,w)+\frac{(v-1)D(v,x)}{v}+E_3(v,w)D(v/w,wx),
\end{multline}
where 
$$
E_1(v,w)=\frac{v(w-1)}{w-v}, \quad E_2(v,w)=\frac{x(1-w)(1-v)}{v(1-x)}
$$
and 
$$
E_3(v,w)=\frac{(1-v)(w-v-vw-w^2x)}{wv(1-x)}.
$$
Let $w_1$ and $w_2$ denote the two roots of the equation~\eqref{eq:w1}. Substituting $w_1$ and $w_2$ into~\eqref{eq:alp2} then gives
$$
a(x;v,1)=E_1(v,w_1)a(x;1,v)+E_2(v,w_1)+\frac{(v-1)D(v,x)}{v}+E_3(v,w_1)D(v/w_1,w_1x)
$$
and
$$
a(x;v,1)=E_1(v,w_2)a(x;1,v)+E_2(v,w_2)+\frac{(v-1)D(v,x)}{v}+E_3(v,w_2)D(v/w_2,w_2x).
$$
Solving  this system of equations yields 
\begin{align*}
a(x;1,v)=\frac{E_2(v,w_1)-E_2(v,w_2)}{E_1(v,w_2)-E_1(v,w_1)}+\frac{E_3(v,w_1)D(v/w_1,w_1x)-E_3(v,w_2)D(v/w_2,w_2x)}{E_1(v,w_2)-E_1(v,w_1)}
 \end{align*}
 Involving expression~\eqref{alg:D} for $D(u,x)$,  
 $$
 w_1+w_2=\frac{v+1}{1-vx}\quad\text{and}\quad w_1w_2=\frac{v}{1-vx},
 $$
 we get~\eqref{exp:110,102} after simplification. This completes the proof of the first statement. 
 
 The second statement follows from the fact that 
 $$1+\sum_{n\geq1}|\I_n(110,102)|x^n=a(x;1,1)+\frac{1}{1-x}$$ equals the  right-hand side of~\eqref{gen:sava}, which ends the proof of the theorem. 
\end{proof}

\subsection{Pattern pair $(102,120)$}\label{sec:4.3}
As for the pattern pair $(102,110)$, we consider the set $\B_{n,m,\ell}$  of $(102,120)$-avoiding inversion sequences of length $n$ in which the largest entry is $m$ with the left-most occurrence  of $m$ in position $\ell$. Denote by $b_{n,m,\ell}$ the cardinality of $\B_{n,m,\ell}$. The first few values of the arrays $[b_{n,m,\ell}]_{0\leq m<\ell\leq n}$ are:
\[ \left[\begin{array}{c}
1
\end{array}\right],
\left[\begin{array}{c c }
1 & 0 \\
0 & 1
\end{array}\right], \left[\begin{array}{c c c}
1 & 0 & 0\\
0 & 2 & 1\\
0 & 0 & 2 
\end{array}\right], \left[\begin{array}{c c c c}
1 & 0 & 0 & 0\\
0 & 4 & 2 & 1 \\
0 & 0 & 5 & 3 \\ 
0 & 0 & 0 & 5
\end{array}\right], \left[\begin{array}{c c c c c}
1 & 0 & 0 & 0 & 0\\
0 & 8 & 4 & 2 & 1 \\
0 & 0 & 13 & 7 & 4 \\ 
0 & 0 & 0 & 14 & 9\\
0 & 0 & 0 & 0 & 14
\end{array}\right].
\] 
The following recurrence relation for $b_{n,m,\ell}$ holds. 

\begin{lemma}For $1\leq m<\ell\leq n$, we have 
\begin{equation}\label{rec:120,102}
b_{n,m,\ell}=d_{\ell,m-1}+\sum_{j=1}^m\sum_{k=\ell}^{n-1}b_{n-1,j,k}.
\end{equation}
\end{lemma}
\begin{proof}
For an element $e\in\B_{n,m,\ell}$, we distinguish two cases:
\begin{itemize}
\item[(a)] If there are at least two $m$ appearing as entries of $e$, then the left-most $m$ is extraneous concerning avoidance of $102$ and $120$.  Deletion of this $m$ results in a member of $\B_{n,m,i}$ for some $\ell\leq i\leq n-1$. So there are $\sum_{i=\ell}^{n-1}b_{n-1,m,i}$ many members of $\B_{n,m,\ell}$ arise  in this case. 
\item[(b)] Otherwise, the letter $m$ occurs only once within $e$. Suppose that $j$ is the second largest  letter of $e$, then we need further to consider two cases:
\begin{itemize}
\item[(b1)] The left-most occurrence of $j$ within $e$ is in positions $1$ through $\ell-1$. Since $e$ is $102$-avoiding, the subsequence formed by entries to the left of $m$ must  be weakly increasing.  Moreover, as $e$ is $120$-avoiding, the entries appear after $e$ are all equal $j$. So members of $\B_{n,m,\ell}$ arising in this case  (with $j$ as second largest entry) are in bijection with Dyck paths of length $\ell-1$ with last east step of height $j$, $0\leq j\leq m-1$. Summing over all $0\leq j\leq m-1$, there are in total   $d_{\ell,m-1}$ many members of $\B_{n,m,\ell}$ arising  in this case. 
\item[(b2)]  The left-most occurrence of $j$ within $e$ is in positions $\ell+1$ through $n$. In each case, the entry $e_{\ell}=m$ is extraneous since the first $\ell-1$ entries of $e$ are governed by a $j$  at some position after $m$. The deletion of $e_{\ell}$ from $e$ then results in a member of $\B_{n,j,k}$ for some $1\leq j\leq m-1$ and $\ell\leq k\leq n-1$. Summing over these indices gives $\sum_{j=1}^{m-1}\sum_{k=\ell}^{n-1}b_{n-1,j,k}$. 
\end{itemize}
\end{itemize}
The above cases together gives~\eqref{rec:120,102}, as desired. 
\end{proof}

For any $n\geq2$, let 
$$
b_{n}(u,v):=\sum_{\ell=2}^n\sum_{m=1}^{\ell-1}b_{n,m,\ell}u^mv^{\ell}.
$$ 
Introduce the generating function $b(x;u,v)$ by
$$
b(x;u,v):=\sum_{n\geq2}b_n(u,v)x^n=uv^2x^2+(2uv^2+uv^3+2u^2v^3)x^3+\cdots.
$$
We could translate recurrence relation~\eqref{rec:120,102} into a functional equation for $B(x;u,v)$.
\begin{lemma}
We have 
\begin{multline}\label{fun:120,102}
b(x;u,v)=\frac{u(D(u,vx)-vx-vxD(1,uvx))}{1-x}+\frac{vx}{(1-v)(1-uv)}b(x;uv,1)\\-\frac{vx}{(1-u)(1-v)}b(x;u,v)+\frac{uvx}{(1-u)(1-uv)}b(x;1,uv).
\end{multline}
\end{lemma}
\begin{proof}
First we compute
\begin{align*}
&\quad\sum_{n\geq2}x^n\sum_{\ell=2}^nv^{\ell}\sum_{m=1}^{\ell-1}d_{\ell,m-1}u^m\\
&=u\sum_{n\geq2}x^n\sum_{\ell=2}^nv^{\ell}(d_{\ell}(u)-C_{\ell-1}u^{\ell-1})\\
&=u\sum_{\ell\geq2}v^{\ell}(d_{\ell}(u)-C_{\ell-1}u^{\ell-1})\sum_{n\geq\ell}x^n\\
&=\frac{u(D(u,vx)-vx-vxD(1,uvx))}{1-x}.
\end{align*}
Next we calculate 
\begin{align*}
&\quad\sum_{n\geq2}x^n\sum_{\ell=2}^nv^{\ell}\sum_{m=1}^{\ell-1}u^m\sum_{j=1}^m\sum_{k=\ell}^{n-1}b_{n-1,j,k}\\
&=\sum_{n\geq2}x^n\sum_{\ell=2}^nv^{\ell}\sum_{j=1}^{\ell-1}\sum_{k=\ell}^{n-1}b_{n-1,j,k}\sum_{m=j}^{\ell-1}u^m\\
&=\sum_{n\geq2}x^n\sum_{j=1}^{n-1}\sum_{\ell=j+1}^nv^{\ell}\sum_{k=\ell}b_{n-1,j,k}\frac{u^j-u^{\ell}}{1-u}\\
&=\frac{1}{1-u}\sum_{n\geq2}x^n\sum_{k=2}^{n-1}\sum_{j=1}^{k-1}\biggl(b_{n-1,j,k}u^j\frac{v^{j+1}-v^{k+1}}{1-v}-b_{n-1,j,k}\frac{(uv)^{j+1}-{uv}^{k+1}}{1-uv}\biggr)\\
&=\frac{1}{1-u}\sum_{n\geq3}x^n\biggl(\frac{vb_{n-1}(uv,1)-vb_{n-1}(u,v)}{1-v}-\frac{uvb_{n-1}(uv,1)-uvb_{n-1}(1,uv)}{1-uv}\biggr)\\
&=\frac{vx}{(1-v)(1-uv)}b(x;uv,1)-\frac{vx}{(1-u)(1-v)}b(x;u,v)+\frac{uvx}{(1-u)(1-uv)}b(x;1,uv). 
\end{align*}
By the above two expressions, multiplying both sides of~\eqref{rec:120,102} by $x^nu^mv^{\ell}$ and summing  over $1\leq m<\ell\leq n$ gives the functional equation~\eqref{fun:120,102}. 
\end{proof}

We are going to solve the functional equation~\eqref{fun:120,102}, which will result in the following expression for $b(x,1,1)$. 

\begin{theorem}We have 
\begin{equation}\label{exp:120,102}
b(x;1,1)=\frac{x^2+x^2\sqrt{1-4x}}{(x-1)((3x-1)\sqrt{1-4x}-4x^2+5x-1)}.
\end{equation}
Equivalently, 
$$
|\I_n(120,102)|=1+\sum_{i=1}^{n-1}{2i\choose i-1}.
$$
\end{theorem}
\begin{proof}
The proof is similar to that of Theorem~\ref{thm:110,102}.
We apply the kernel method and set the coefficients of $b(x;u,v)$ on both sides of~\eqref{fun:120,102} equal:
$$
1=-\frac{vx}{(1-u)(1-v)}.
$$
Solving this equation for $u$ gives 
\begin{equation}\label{eq:u2}
u=u(x,v)=\frac{1-v+vx}{1-v}
\end{equation}
Let us denote 
$$
\beta=\beta(x,v)=uv=\frac{v-v^2+v^2x}{1-v}.
$$
Substituting the expression for $u$ given by~\eqref{eq:u2} into~\eqref{fun:120,102},  and simplifying yields 
\begin{multline}\label{eq:beta}
b(x;1,\beta)=\frac{\beta-v}{\beta-\beta v}b(x;\beta,1)+\frac{(v-\beta)(1-\beta))}{v^2x(x-1)}D(\beta/v,vx)\\+\frac{(v-\beta)(1-\beta))}{v(1-x)}(D(1,\beta x)+1). 
\end{multline}
Observe that $\beta(x,w)=v$ whenever $w=w(x,v)$ is a root of the quadratic equation
\begin{equation}\label{eq:w2}
(x-1)w^2+(1+v)w-v=0.
\end{equation}
Replacing  $v$ by $w$ in~\eqref{eq:beta} gives
\begin{multline}\label{eq:beta2}
b(x;1,v)=\frac{v-w}{v-vw}b(x;v,1)+\frac{(w-v)(1-v))}{w^2x(x-1)}D(v/w,wx)\\+\frac{(w-v)(1-v))}{w(1-x)}(D(1, vx)+1). 
\end{multline}
substituting the two roots of the quadratic equation~\eqref{eq:w2} into~\eqref{eq:beta2} yields a system of two equations, which after eliminating $b(x;1,v)$ and then setting $v=1$ gives the desired  expression~\eqref{exp:120,102} for $b(x;1,1)$. This completes the proof of the theorem. 
\end{proof}
 
 \section{Classes 4140(D,E,F): Bell numbers and Stirling numbers}
 \label{sec:bell}

Let $\Pi_n$ be the set of all set partitions of $[n]$ and denote $\Pi_{n,k}$ the set of partitions in $\Pi_n$ with $k$ blocks. It is well known that 
$$
|\Pi_n|=B_n\quad\text{and}\quad|\Pi_{n,k}|=S(n,k),
$$
where $B_n$ is the $n$-th {\em Bell number} and $\{S(n,k)\}_{1\leq k\leq n}$ is the triangle of  {\em Stirling numbers of the second kind}. The Stirling numbers $S(n,k)$ satisfy the following basic recurrence (cf.~\cite[Sec.~1.4]{st}):
$$
S(n,k)=kS(n-1,k)+S(n-1,k-1).
$$
It was shown by Corteel et al.~\cite[Theorem~11]{cor} that 

\begin{theorem}[Corteel--Martinez--Savage--Weselcouch]\label{cor:bell}
The number of $011$-avoiding inversion sequences in $\I_n$ with $k$  zeros is $S(n,k)$. 
\end{theorem}
The proof of this result in~\cite{cor} was by constructing a recursive bijection from $\I_{n,k}(011)$ to $\Pi_{n,k}$, where $\I_{n,k}(011):=\{e\in\I_n(011):\zero(e)=k\}$. We note that there exists another  natural bijection $\eta: \I_{n,k}(011)\rightarrow\Pi_{n,k}$ observed in~\cite{chen}. Given $e\in\I_{n,k}(011)$, the partition $\eta(e)$ can be constructed for $i=1,2,\ldots,n$, step by step, as follows: 
\begin{itemize}
\item if $e_i=0$, then construct a new block $\{i\}$;
\item otherwise, $e_i>0$ and insert $i$ into the  block containing $e_i$. 
\end{itemize}
For example, if $e=(0,0,2,1,0,4)\in\I_{6,3}(011)$, then $\eta(e)=\{\{1,4,6\},\{2,3\},\{5\}\}\in\Pi_{6,3}$. 
The bijection $\eta$ provides another proof of Theorem~\ref{cor:bell} . 

A statistic whose distribution over a combinatorial object gives the Stirling numbers of the second kind is said to be  {\em 2nd Stirling} over such object. In this language, Theorem~\ref{cor:bell} is equivalent to the assertion that ``$\zero$'' is 2nd Stirling  over $\I_n(011)$. Since an inversion sequence contains the pattern $101$ or $110$ must also contain the pattern $011$, we have 

\begin{corollary}
For $n\geq1$, $\I_n(011,101)=\I_n(011)=\I_n(011,110)$. 
\end{corollary}

It has been shown in~\cite[Sec.~6]{cjl} that $|\I_n(000,101)|=|\I_n(000,110)|=B_n$ and moreover, the statistic ``$\rmin$'' is 2nd Stirling over $|\I_n(000,110)|$. Unfortunately, all the five classical statistics defined in the introduction are not 2nd Stirling over $\I_n(000,101)$. 
We pose the following challenging open problem for further research. 

\begin{?}
Find a natural statistic whose distribution  over $\I_n(000,101)$ is 2nd Stirling. 
\end{?}

For a sequence $e\in\I_n$, let $\rep(e):=n-\dist(e)$ be the number of repeated entries of $e$. 
Martinez and Savage~\cite[Sec.~2.17.3]{ms} showed that $|\I_n(010,101)|=B_n$ by constructing a bijection between $\I_n(010,101)$ and $\I_n(011)$. We have the following refinement. 

\begin{theorem}\label{rep:stir}
The statistic ``$\rep$'' is  2nd Stirling over $\I_n(010,101)$. 
\end{theorem}
\begin{proof}
The mapping $e\mapsto e'$ constructed in~\cite{ms}, where $e_i'=0$ if $e_i\in\{e_1,e_2,\ldots,e_{i-1}\}$ and otherwise $e_i'=e_i$, is a bijection from $\I_n(010,101)$ and $\I_n(011)$ which preserves the statistic ``$\dist$''. Since each $011$-avoiding inversion sequence has distinct positive entries, the statistics ``$\rep$'' and ``$\zero$'' coincide on $\I_n(011)$. The result then follows from Theorem~\ref{cor:bell}. 
\end{proof}

\begin{theorem}
There exists a bijection $\phi:\I_n(010,101)\rightarrow\I_n(010,100)$ which preserves the triple $(\dist,\satu,\zero)$. 
Consequently, the statistic  ``$\rep$'' is  2nd Stirling over  $\I_n(010,100)$. 
\end{theorem}

\begin{proof}
Note that $e\in\I_n$ is $(010,101)$-avoiding if and only if whenever $e_i=e_j$ for some $i<j$, then $e_k=e_i$ for every $i\leq k\leq j$.   For $e\in\I_n(010,101)$, define $\phi(e)=e'$ where  $e_i'=\max\{e_1,e_2,\ldots,e_{i-2}\}$ if $e_{i-1}=e_i<\max\{e_1,e_2,\ldots,e_{i-2}\}$ and otherwise $e_i'=e_i$. For instance, if $$e=(0,0,0,0,4,3,2,2,2,5,1,1)\in\I_{12}(010,101),$$ then 
$$\phi(e)=(0,0,0,0,4,3,2,4,4,5,1,5)\in\I_{12}(010,100).$$
It is routine to check that $\phi$ is a bijection between $\I_n(010,101)$ and $\I_n(010,100)$ that preserves the triple $(\dist,\satu,\zero)$. The second statement then follows from the bijection $\phi$ and Theorem~\ref{rep:stir}. 
\end{proof}
 
\section{Class 5798B: A bijection to weighted ordered trees}
\label{sec:tree}
The integer sequence 
$$
\{1,1,2,6,21,80,322,1347,5798,25512,\ldots\}
$$
with algebraic generating function 
$$
A(t)=1+\frac{tA(t)}{1-tA(t)^2}
$$
appearing as A106228 in the OEIS~\cite{oeis} enumerates many interesting combinatorial objects, including 
\begin{itemize}
\item $(4123,4132,4213)$-avoiding permutations, proved by Albert,  Homberger,  Pantone, Shar and Vatter~\cite{ahpsv};
\item $(101,102)$-avoiding inversion sequences,  conjectured by Martinez and Savage~\cite{ms} and verified algebraically  by Cao, Jin and Lin~\cite{cjl};
\item weighted ordered trees, where each {\em interior vertex} (non-root, non-leaf) is weighted by a positive integer less than or equal to its outdegree (see~\cite[A106228]{oeis}). 
\end{itemize}
In this section, we will prove that this integer sequence also counts $(101,021)$-avoiding inversion sequences. This is achieved via a new constructed ``type''-preserving  bijection from ordered trees of $n$ edges to Dyck paths of length $n$.

Let ${\bf t}=(t_1,t_2,\ldots,t_k)$ be a sequence of positive integers  with $t_1+t_2+\cdots+t_k=n$. A Dyck path $D=h_1h_2\cdots h_n$ of length $n$ is said to have {\em type ${\bf t}$} if 
$$
h_i\neq h_{i+1}\Longleftrightarrow i\in\{t_1,t_1+t_2,\ldots,t_1+t_2+\cdots+t_{k-1}\}.
$$
 In other words, this Dyck path begins with $t_1$ east steps followed by at least one north step,  and then continues with $t_2$ east steps followed by at least one north step, and then so on. 
 For example, the Dyck path in the right-bottom side of Fig.~\ref{tree:dyck} has type $(3,1,2,1)$. 
 
We next introduce the type for ordered trees. For the basic definitions and terminology concerning ordered trees (or plane trees), see~\cite[Sec.~1.5]{st2}. In an ordered tree $T$, the outdegree of a vertex $v\in V(T)$, denoted $\deg_v$, is the number of descendants of $v$. We will order all vertices of $T$ by using the {\em depth-first order} (also known as {\em preorder}). For example, the preorder of the vertices of the ordered tree in the right-top side of Fig.~\ref{tree:dyck} are $a,b,c,d,e,f,g,h$. If $v_1$ is the root of $T$ and $v_2,v_3,\ldots,v_k$ are the interior vertices of $T $ in preorder, then we say $T$ has type $(\deg_{v_1},\deg_{v_2},\ldots,\deg_{v_k})$. For instance, the tree in our running example has type $(\deg_a,\deg_c,\deg_d,\deg_g)=(3,1,2,1)$. Let $\T_n$ be the set of all ordered trees with $n$ edges. It is well known that $|\T_n|=C_n$ and there are already several bijections between $\T_n$ and $\D_n$; see the monograph~\cite{st2}  of  Stanley on Catalan number.  The main result of this section is a construction of a bijection from $\T_n$ to $\D_n$, which seems new to the best of our knowledge. 
 
  \begin{figure}
\begin{tikzpicture}[scale=0.4]
\draw[-] (3,14) to (1,12) (3,14) to (3,12) (3,14) to (5,12);
\node at (3,14) {\circle*{4}};
\node[color=red] at (1,12) {\circle*{4}};
\node[color=red] at (3,12) {\circle*{4}};\node at (3,12) {\circle{5.5}};
\node[color=red] at (5,12) {\circle*{4}};

\node at (8,12) {$\longrightarrow$};
\node at (18,12) {$\longrightarrow$};
\node at (28,12) {$\longrightarrow$};
\node at (33,8.7) {$\big\downarrow$};
\node at (33.8,8.7) {$\Phi$};
\draw[-] (13,14) to (11,12) (13,14) to (13,12) (13,14) to (15,12) (13,12) to (13,10);
\node at (13,14) {\circle*{4}};
\node at (11,12) {\circle*{4}};
\node at (13,12) {\circle*{4}};
\node[color=red] at (15,12) {\circle*{4}};\node at (13,10) {\circle{5.5}};
\node[color=red] at (13,10) {\circle*{4}};
\draw[-] (23,16) to (21,14) (23,16) to (23,14) (23,16) to (25,14) (23,14) to (23,12)
(23,12) to (21,10) (23,12) to (25,10);
\node at (23,16) {\circle*{4}};
\node at (21,14) {\circle*{4}};
\node at (23,14) {\circle*{4}};
\node[color=red] at (25,14) {\circle*{4}};
\node at (23,12) {\circle*{4}};
\node[color=red] at  (21,10) {\circle*{4}};
\node[color=red] at  (25,10) {\circle*{4}};\node at  (25,14) {\circle{5.5}};
\draw[-] (33,16) to (31,14) (33,16) to (33,14) (33,16) to (35,14) (33,14) to (33,12)
(33,12) to (31,10) (33,12) to (35,10) (35,14) to (35,12);
\node at (33,16) {\circle*{4}}; \node at (33.5,16.2) {$a$};
\node at (31,14) {\circle*{4}}; \node at (30.5,14) {$b$};
\node at (33,14) {\circle*{4}};  \node at (32.5,14) {$c$};
\node at (35,14) {\circle*{4}};  \node at (35.5,14) {$g$};
\node at (33,12) {\circle*{4}};  \node at (32.5,12.2) {$d$};
\node at  (31,10) {\circle*{4}}; \node at (31.5,10) {$e$};
\node at  (35,10) {\circle*{4}}; \node at (35.5,10) {$f$};
\node[color=red] at  (35,12) {\circle*{4}}; \node at (35.5,12) {$h$};

\draw[step=1,color=gray](0,0) grid (7,7);
\draw[color=blue](0,0)--(7,7);
\draw[very thick,color=red](0,0)--(3,0);

\node at (8.4,3.5) {$\longrightarrow$};
\node at (18.4,3.5) {$\longrightarrow$};
\node at (28.4,3.5) {$\longrightarrow$};
\draw[step=1,color=gray](10,0) grid (17,7);
\draw[color=blue](10,0)--(17,7);
\draw[very thick](10,0)--(13,0);
\draw[very thick,color=red](13,0)--(13,2) (13,2)--(14,2);
\draw[step=1,color=gray](20,0) grid (27,7);
\draw[color=blue](20,0)--(27,7);
\draw[very thick](20,0)--(23,0) (23,0)--(23,2) (23,2)--(24,2);
\draw[very thick,color=red](24,2)--(24,3)  (24,3)--(26,3);
\draw[step=1,color=gray](30,0) grid (37,7);
\draw[color=blue](30,0)--(37,7);
\draw[very thick](30,0)--(33,0) (33,0)--(33,2) (33,2)--(34,2) (34,2)--(34,3) (34,3)--(36,3);
\draw[very thick,color=red](36,3)--(36,6)  (36,6)--(37,6)--(37,7);
\end{tikzpicture}
\caption{An example of the construction of  the bijection $\Phi$\label{tree:dyck}}
\end{figure}
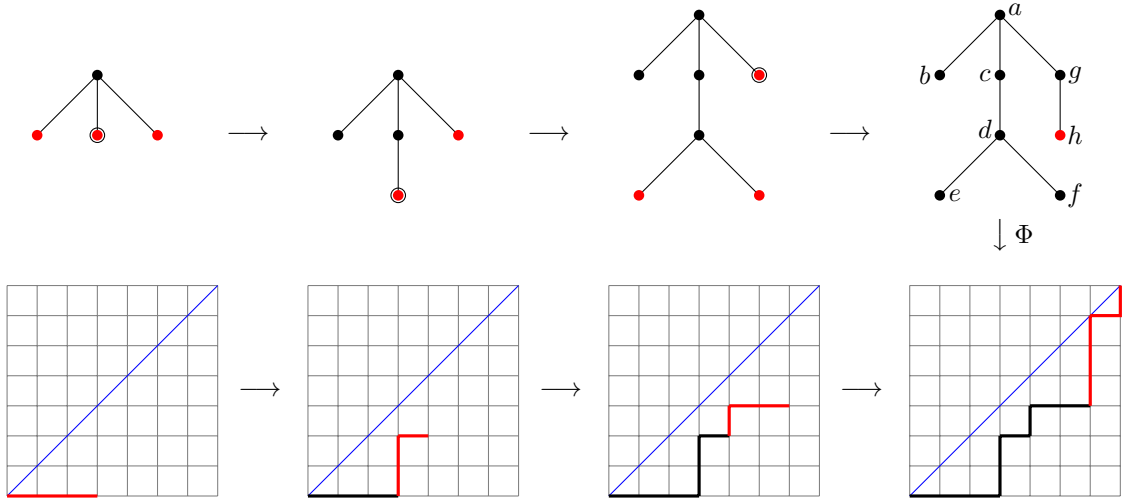

\begin{theorem}\label{thm:tree}
There exists a type-preserving bijection $\Phi$ from $\T_n$ to $\D_n$.
\end{theorem}

Recall that in Section~\ref{sec:4.1},  the geometric representation $\mathcal{H}$ is a bijection between $\I_n(021)$ and $\A_n$. This $h$ reduces to a bijection between $\I_n(101,021)$ and colored Dyck paths in $\A_n$ such that red east steps of the same height are connected. The latter objects are readily seen to be in one-to-one correspondence with Dyck paths in $\D_n$ such that each first appearance of east steps of the same height is weighted by a positive integer less than or equal to the number of east steps of that height. The later objects are  in bijection with  weighted ordered trees where each  interior vertex is weighted by a positive integer less than or equal to its outdegree because of Theorem~\ref{thm:tree}. This proves the following enumerative result. 

\begin{corollary}
Inversion sequences avoiding the pair $(101,021)$ are counted by the integer sequence A106228. Equivalently, $|\I_n(101,102)|=|\I_n(101,021)|$ for $n\geq1$. 
\end{corollary}

The construction of $\Phi$ will be recursive based on the number of interior vertices of rooted trees. For an ordered tree $T\in\T_n$ with $v$ as the latest (in preorder) interior vertex, define $\capp(T)$ to be the number of vertices of $T$ after $v$ (in preorder), called the {\em capacity of $T$}. While for a Dyck path (or just a word of integers) $D=h_1h_2\cdots h_n$, define $\capp(D)=n-h_n$ the {\em capacity of $D$}. We now construct $\Phi(T)$ recursively as follows:
\begin{itemize}
\item if $T$ has no interior vertex, then define $\Phi(T)=00\cdots0$, the Dyck path consisting of $n$  east steps of height $0$;
\item otherwise, suppose $T$ has at least one interior vertex and $v$ is the latest (in preorder) interior vertex; let $T'$ be the ordered tree obtained by deleting all descendants of $v$ in $T$ and let $\Phi(T')=h'_1h'_2\cdots h'_{\ell}$, where $\ell=n-\deg_v$; let $v'$ be the latest interior vertex of $T'$ (if $T'$ has no interior vertex, then set $v'$ to be the root of $T'$) and suppose that $v$ is the $h$-th leaf after $v'$ (in preorder) in $T'$ with $1\leq h\leq \capp(T')$; define 
$$\Phi(T)=h'_1h'_2\cdots h'_{\ell}h_{\ell+1}h_{\ell+2}\cdots h_n,$$
where $h_{\ell+1}=h_{\ell+2}=\cdots=h_n=h'_{\ell}+h$. 
\end{itemize}
See Fig.~\ref{tree:dyck} for an example of the construction of $\Phi(T)$. 

The following crucial property of $\Phi$ can be checked routinely  by induction on the number of interior vertices of $T$.

\begin{lemma}
The mapping $\Phi$ is capacity preserving, namely, $\capp(T)=\capp(D)$ with $D=\Phi(T)$ considered as a word of integers. 
\end{lemma}

It follows from the above lemma that the resulting word $\Phi(T)$ is a Dyck path and so the mapping $\Phi$ is well-defined. It can be checked from the recursive construction that $\Phi$ is a type-preserving injection and thus is a bijection  due to the cardinality reason.

\section{Class 11624 and indecomposable set partitions}
\label{sec:irre}

A set partition of $[n]$ is a {\em indecomposable partition} (or {\em unsplitable partition}) if the only subset of its blocks that partitions an initial segment of $[n]$ is the full set of all blocks. For our purpose, we  arrange the blocks of a partition in decreasing order of their smallest entry, as in $\{7\}\{4\}\{3,5\}\{1,2,6\}$. Since the last three blocks form a partition of $[6]$, this partition is not indecomposable. The indecomposable partitions can be used to index a free generating set of symmetric functions in noncommutative variables (see~\cite{clw}). 

The indecomposable partitions are counted by the integer sequence~\cite[A074664]{oeis}, which has several other combinatorial interpretations, including irreducible set partitions and non-nesting permutations. In this section, we prove that this integer sequence also counts $(101,110)$-avoiding inversion sequences. In order to achieve this, we need to introduce a refinement of this integer sequence. 

The triangle $\{T_{n,k}\}_{1\leq k\leq n}$,  appearing as A086211 in OEIS~\cite{oeis}, are defined by 
\begin{equation}
T_{n,k} =T_{n-1,k-1} +kT_{n-1,k} + \sum_{j=k+1}^{n-1}T_{n-1,j} 
\end{equation}
with $T_{1,i}=\delta_{1,i}$.  The first few values of $T_{n,k}$ are listed in Table~\ref{val:tnk}.
 \begin{table}
\label{Lnkz}
 \[ \begin{tabular}{c|cccccc}
$n\backslash k$ & \,\,$1$ & $2$ & $3$ & $4$ & $5$
\\
\hline
 $1$\,\, &\,\,$1$\\ 
$2$\,\, & \,\,$1$ & $1$  \\
$3$\,\, & \,\,$2$ &$3$&$1$  \\
$4$\,\, &\,\,$6$ &  $9$ &$6$ &$1$  \\
$5$\,\, &\,\,$22$ &  $31$ &$28$ &$10$ & $1$  \\
\end{tabular} \]
\caption{The first values of $T_{n,k}$\label{val:tnk}}
\end{table}
As was proven  by Callan~\cite{cal},  $T_{n,k}$ counts the  indecomposable partitions of $[n+1]$ such that $n+1$ lies in the $k$-th block. For example, $T_{3,2}=3$ and all the indecomposable partitions of $[4]$ such that $4$ lies in the second block are $\{2\}\{1,3,4\}$, $\{2,3\}\{1,4\}$ and $\{3\}\{1,2,4\}$. 
The following result provides another interpretation for $T_{n,k}$. 

\begin{theorem}
For $1\leq k\leq n$, we have 
$$
|\{e\in\I_n(101,110):\zero(e)=k\}|=T_{n,k}. 
$$
In particular, the cardinality of $\I_n(101,110)$  equals the number of indecomposable partitions of $[n+1]$. 
\end{theorem}
\begin{proof}
Let 
$$
Z_{n,k}=\{e\in\I_n(101,110):\zero(e)=k\}\quad\text{and}\quad z_{n,k}=|Z_{n,k}|. 
$$
It will be sufficient to show that 
\begin{equation}\label{eq:101,110}
z_{n,k} =z_{n-1,k-1} +kz_{n-1,k} + \sum_{j=k+1}^{n-1}z_{n-1,j}
\end{equation}
with initial condition $z_{1,i}=\delta_{1,i}$. 

Let $Z_{n,k,\ell}$ be the set of  sequences $e\in Z_{n,k}$ such that within $e$ there are exactly $\ell$ ones. For $e\in\I_n$, let us introduce the operation $\sigma_{-1}(e)=(e_2',e_3',\ldots,e_{n}')$, where  $e_i'=e_i-\chi(e_i>0)$. It is clear that $\sigma_{-1}(e)\in Z_{n-1,k+\ell-1}$. In fact, if $\ell=0$, then $\sigma_{-1}$ is a bijection between $Z_{n,k,0}$ and $Z_{n-1,k-1}$ and so $|Z_{n,k,0}|=z_{n,-1,k-1}$. However, if $\ell=1$, each element of $Z_{n-1,k}$ is the image under $\sigma_{-1}$ of  $k$ elements of $Z_{n,k,1}$; and if $\ell\geq2$, then each element of $Z_{n-1,k+\ell-1}$ is the image under $\sigma_{-1}$ of an unique  element from $Z_{n,k,\ell}$. To see this, for $e\in Z_{n,k,\ell}$, let $\tilde{e}$ be the subsequence of $e$ consisting of all of  the zeros and ones in $e$. Since $e$ avoids both $101$ and $110$, when $\ell\geq2$, the $\ell$ ones in $\tilde{e}$ must be consecutive and in the last $\ell$ positions of $\tilde{e}$. But if $\ell=1$, then there are $k$ ways to place the entry one in $\tilde{e}$. It follows that 
$$
|Z_{n,k,1}|=kz_{n-1,k}\quad\text{and}\quad |Z_{n,k,\ell}|=z_{n-1,k+\ell-1}\,\,\text{for $\ell\geq2$}. 
$$
Taking all the above cases into account gives 
$$
z_{n,k}=\sum_{\ell\geq0}|Z_{n,k,\ell}|=z_{n-1,k-1} +kz_{n-1,k} + \sum_{\ell\geq2}z_{n-1,k+\ell-1},
$$
which proves~\eqref{eq:101,110}. This completes the proof of the theorem.  
\end{proof}

\section{Wilf-equivalences} 
\label{sec:wilf}
Two sets of patterns $\Pi$ and $\Pi'$ are said to be {\em Wilf-equivalent} on $\I_n$
if for each positive integer $n$, 
$$|\I_n(\Pi)|=|\I_n(\Pi')|.$$ 
We denote the Wilf-equivalence of $\Pi$ and $\Pi'$ by $\Pi\thicksim\Pi'$. 
Our results in previous sections together with the works in~\cite{bgrr,ms,cjl}  have already established most of the Wilf-equivalences for inversion sequences avoiding pairs of  patterns of length $3$; see Tables~\ref{invseq:pairs} and~\ref{invseq:pairs2}. The computer data indicates that there are exactly  $48$ Wilf-equivalence classes, of which $4$ classes are remain unproved.  
The purpose of this section is to prove these remaining $4$ classes, which would completely classify all the Wilf-equivalences for inversion sequences avoiding pairs of  patterns of length $3$. 

We will apply a  bijection $\varphi: \I_n(210)\rightarrow\I_n(201)$ of Corteel et al.~\cite[Thm.~5]{cor}. For the sake of convenience, we review the construction of $\varphi$ as follows. It was proven  in~\cite[Obs.~3]{cor}  that $210$-avoiding inversion sequences are precisely those that can be partitioned into two weakly increasing subsequences. Given $e\in\I_n(210)$, suppose that $e_{a_1}\leq e_{a_2}\leq\cdots\leq e_{a_t}$ is the sequence of weak left-to-right maxima of $e$ and $e_{b_1}\leq e_{b_2}\leq\cdots\leq e_{b_{n-t}}$ is the subsequence of remaining elements of $e$. Define $\varphi(e)=(f_1,f_2,\ldots,f_n)$, where 
\begin{itemize}
\item  $f_{a_i}=e_{a_i}$ for $i=1,\ldots,t$ and 
\item for each $j=1,2,\ldots,n-t$, extract an element of the multiset $B=\{e_{b_1},e_{b_2},\ldots,e_{b_{n-t}}\}$ and assign it to $f_{b_1},f_{b_2},\ldots,f_{b_{n-t}}$ according to the rule:  
$$
f_{b_j}=\max\{k\,|\, k\in B-\{f_{b_1},f_{b_2},\ldots,f_{b_{j-1}}\}\text{ and } k<\max(e_1,\ldots,e_{b_{j-1}})\}. 
$$
\end{itemize}
As an example, we have 
\begin{equation}\label{exm:varphi}
e=(\blue{0,0,0,3},1,\blue{4},2,2) \mapsto\varphi(e)=(\blue{0,0,0,3},2,\blue{4},2,1),
\end{equation}
where the weak left-to-right maxima of $e$ are colored blue. 
\begin{theorem}\label{wilf:3}The following three Wilf-equivalences hold:
$$
(011,201)\thicksim(011,210),\quad (000,201)\thicksim(000,210)\quad\text{and}\quad (100,021)\thicksim(110,021).  
$$
\end{theorem}

\begin{proof}
An inversion sequence is $011$-avoiding if its positive entries are distinct, and is $000$-avoiding if no entry occurs more than twice.
Notice that the bijection  $\varphi$ between $\I_n(210)$ and $\I_n(201)$ is just rearrangement of the entries of inversion sequences. Thus, $\varphi$ preserves both the patterns $011$ and $000$.
It then follows that $\varphi$ restricts to a bijection between $\I_n(011,210)$ (resp.~$\I_n(000,210)$) and $\I_n(011,201)$  (resp.~$\I_n(000,201)$). This proves the first two Wilf-equivalences. 

For the third Wilf-equivalence, define the mapping $e\mapsto e'$ where $e\in\I_n(100,021)$, by 
\begin{itemize}
\item $e'=e$, if the entries of $e$ are weakly increasing;
\item otherwise, $e$ contains exactly one zero entry, say $e_k=0$, after the first positive entry. Then let $e'$ be the inversion sequence obtained from $e$ by changing all positive entries before $e_k$ to zeros, except all these that are the first appearances of each positive value. 
\end{itemize}
For instance, we have 
$$e=(0,0,1,1,2,2,2,0,3,3,5)\mapsto e'=(0,0,1,0,2,0,0,0,3,3,5).
$$
It is routine  to check that this mapping establishes a one-to-one correspondence between $\I_n(100,021)$ and $\I_n(110,021)$. 
\end{proof}

Example~\eqref{exm:varphi} indicates that  $\varphi$ does not restrict to a bijection between $\I_n(010,201)$ and $\I_n(010,210)$, because it  does not preserve the pattern $010$.  However, we still have the following equinumerosity. 

\begin{theorem}
There exists a bijection $\psi:\I_n(201)\rightarrow\I_n(210)$ which preserves the pattern $010$ and the triple of statistics $(\zero,\dist,\satu)$. In particular, $(010,201)\thicksim(010,210)$.
\end{theorem}

\begin{proof}
The idea underlying the construction of $\psi$ is to replace iteratively, occurrences of the pattern $210$ in an element of $\I_n(201)\setminus\I_n(210)$ with copies of the pattern $201$, preserving the pattern $010$.

We call a subsequence $e_ae_be_c$ ($a<b<c$) of an inversion sequence $e$ that is order isomorphic to $210$ (resp.~$201$) a  $k$-occurrence of $210$ (resp.~$201$)  if $e_a=k$.
Since $\psi$ is simply the identity on  $\I_n(201)\cap\I_n(210)$, we only need to define the map $\psi$ from $\I_n(201)\setminus\I_n(210)$ to $\I_n(210)\setminus\I_n(201)$. 

We need to introduce a fundamental operation on an inversion sequence $e$. Suppose $S$ is a subsequence of $e$ and the set of distinct entries in $S$ is $\{v_1,v_2,\ldots,v_{\ell}\}$ with $v_1<v_2<\cdots<v_{\ell}$ for some $\ell\geq2$. Introduce the {\bf cyclic exchange} on $e$ with respect to $S$ according to the following two cases: 
\begin{itemize}
\item if $v_1=0$, then rearrange  the entries of $S$  so that all entries of $S$ appear in weakly increasing order;
\item otherwise, change all entries of $S$ with value $v_t$ to $v_{t+1}$ for $1\leq t\leq \ell-1$ and change all entries of $S$ with value $v_{\ell}$ to $v_1$. 
\end{itemize}
Denote by $\mathcal{C}(e,S)$ the resulting sequence (not necessary an inversion sequence in general). 
For instance, if $e=(0,0,0,3,\blue{2,2},4,\blue{1},3,\blue{0,1,1})$ (resp.~$(0,0,0,3,\blue{1},4,\blue{3,1},3,\blue{1,2,2})$) with the subsequence $S$ in blue, then 
$$
\mathcal{C}(e,S)=(0,0,0,3,\blue{0,1},4,\blue{1},3,\blue{1,2,2})\quad\text{(resp.~$(0,0,0,3,\blue{2},4,\blue{1,2},3,\blue{2,3,3})$)}.
$$
Since whenever  $0$ is a letter of $S$, the sequence $\mathcal{C}(e,S)$ is just rearrangement of $e$, the operation $\mathcal{C}$ preserves the number of zero entries.  It is  clear that the operation $\mathcal{C}$ also preserves the number of distinct entries of inversion sequences.

Let $e\in\I_n(201)\setminus\I_n(210)$. Then, $e$ contains the pattern $210$ but avoids  $201$. 
We apply the following  algorithm on $e$, where $t$ is a temporary variable: 
\begin{itemize}
\item[(1)](Start) suppose the distinct entries greater than one appearing in $e$ are $2\leq k_1<k_2<\cdots<k_m\leq n-1$; set $e^{(0)}=e$, $t\leftarrow 1$ and go to step (2);
\item[(2)] go to step (3) if $e^{(t-1)}$ has  at least one $k_{t}$-occurrence of $210$, otherwise, set $e^{(t)}=e^{(t-1)}$, $t\leftarrow t+1$ and do step (2) again; 
\item[(3)] let $S^{(t)}$ be the subsequence of $e^{(t-1)}$ consisting of entries after the left-most occurrence of $k_{t}$ whose values are smaller than $k_{t}$; set $e^{(t)}=\mathcal{C}(e^{(t-1)},S^{(t)})$, $t\leftarrow t+1$ and go to step (2).
\end{itemize}
Finally, this algorithm terminates when $t= m+1$ and we define $\psi(e)=e^{(m)}$.  For example,  if $e=(0,0,1,2,3,2,2,4,3,4,8,7,5,4,3,0)$
is a sequence in $\I_{15}(201)\setminus\I_{15}(210)$, then 
\begin{align*}
e&=(0,0,1,2,\red{3},\blue{2,2,}4,3,4,8,7,5,4,3,\blue{0})\rightarrow(0,0,1,2,3,0,2,\red{4},\blue{3},4,8,7,5,4,\blue{3},\blue{2})\\
&\rightarrow(0,0,1,2,3,0,1,4,2,4,8,7,\red{5},\blue{4,2,3})\rightarrow (0,0,1,2,3,0,1,4,2,4,8,\red{7},\blue{5,2,3,4})\\
&\rightarrow(0,0,1,2,3,0,1,4,2,4,\red{8},\blue{7,2,3,4,5})\rightarrow(0,0,1,2,3,0,1,4,2,4,8,2,3,4,5,7)=\psi(e).
\end{align*}

We need to show that $\psi$ is well-defined, i.e., $\psi(e)\in\I_n(210)\setminus\I_n(201)$. It  suffices  to prove that  if $e^{(t-1)}$ has no $k_i$-occurrence of $210$ for any $1\leq i\leq t-1$ and contains $k_{t}$-occurrence of $210$, then $e^{(t)}$ has no $k_i$-occurrence of $210$ for any $1\leq i\leq t$ and contains $k_{t}$-occurrence of $201$. To see this, suppose that the distinct entries of $S^{(t)}$ are $v_1<v_2<\cdots<v_{\ell}$, then $v_{\ell}$ must be $k_{t-1}$. Otherwise, $v_{\ell}<k_{t-1}$ and all $k_{t-1}$'s in $e^{(t-1)}$ will appear to the left of the left-most $k_t$, which implies that $e^{(t-1)}$ has $k_{t-1}$-occurrence of $210$, a contradiction. Next, we claim that all $v_{\ell}$'s in $S^{(t)}$ appear before all $v_{i}$'s for $1\leq i\leq \ell-1$. For otherwise, $k_t$ and $k_{t-1}=v_{\ell}$ will play the roles of $2$ and $1$ in a pattern $201$ of $e^{(t-1)}$, which forces $e$ to contain the pattern $201$ because $e^{(t-1)}$ is obtained from $e$ by modifying the entries smaller than $k_{t-1}$. Thus, we come to the crucial conclusion that 
\begin{center}
($\star$) all entries of $S^{(t)}$ must appear from left to right in the order: $v_{\ell}=k_{t-1},v_1,v_2,\ldots,v_{\ell-1}$. 
\end{center}
So after applying the cyclic exchange on $e^{(t-1)}$ with respect to $S^{(t)}$, the entries after the left-most $k_t$ and smaller than $k_t$ appear in weakly increasing order, which removes all $k_t$-occurrences of $210$ and replaces them with $k_t$-occurrences of $201$ in $e^{(t)}$.  This proves that $\psi$ is well-defined. 

In view of the property ($\star$), each cyclic exchange operation involved in step (3) is reversible and so $\psi$ is bijective. We have already known that the cyclic exchange operation $\mathcal{C}$ preserves the pair of statistics $(\zero,\dist)$, and so does $\psi$. Moreover, since in  step (3) of the algorithm of constructing $\psi$, we just modify the entries after some $k_t$ with values smaller than $k_t$,  $\psi$ also preserves the statistic `$\satu$'. Finally, we claim  that if  $e^{(t)}=\mathcal{C}(e^{(t-1)},S^{(t)})$, then $e^{(t-1)}$ contains  $010$ if and only if $e^{(t)}$ contains $010$. If $aba$ ($a<b$) is a $010$-subsequence of $e^{(t-1)}$, then this $aba$ can not be a subsequence of $S^{(t)}$ because of  property ($\star$). We distinguish several cases:
\begin{itemize}
\item[1)] all letters of $aba$ are not in $S^{(t)}$, then  $aba$ is a $010$ pattern of $e^{(t)}$.
\item[2)] both $a$'s are entries of $S^{(t)}$, then 
\begin{itemize}
\item[2a)] if $a=0$, then subsequence $0k_t0$ is a $010$ pattern of $e^{(t)}$, where the left $0$ is the initial zero of $e^{(t)}$ and the right $0$ is a zero entry (must exist) after $k_t$;
\item[2b)] otherwise $a\neq0$, then both $a$'s will be changed into two entries with the same value smaller than $b$ (since $b>k_t$), which together with $b$ forms a $010$ pattern of $e^{(t)}$. 
\end{itemize}
\item [3)] only  the right $a$ is an entry of $S^{(t)}$, then still $aba$ is a $010$-subsequence of $e^{(t)}$. 
\item[4)] $ba$ are subsequence of $S^{(t)}$, then $ak_ta$ is a $010$-subsequence of $e^{(t)}$. 
\end{itemize}
In all cases, $e^{(t)}$ contains the pattern $010$, which proves the `only if' side of the claim. The `if' side of the claim can be proven by similar discussions, which will be omitted. This completes the proof of the theorem. 
\end{proof}

\subsection{Revisiting  $(011,201)$ and $(011,210)$: an unbalanced Wilf-equivalence conjecture}
Following Burstein and Pantone~\cite{bur},  a Wilf-equivalence is called {\em unbalanced} if the two sets of patterns do not contain the same number of patterns of each length. The first two unbalanced Wilf-equivalences for classical patterns  in permutations were proved in~\cite{bur}. For inversion sequences, several instances of unbalanced Wilf-equivalences have been proved by Martinez and Savage~\cite{ms} such as 
$$
|\I_n(021)|=S_n=|\I_n(210,201,101,100)|\quad\text{and}\quad|\I_n(011)|=B_n=|\I_n(000,110)|. 
$$
Regarding the pattern pairs $(011,201)$ and $(011,210)$, we make  the following unbalanced Wilf-equivalence conjecture. 

\begin{conjecture}\label{conj:wilf}
The following  Wilf-equivalences hold:
$$
(011,201)\thicksim(011,210)\thicksim(110,210,120,010)\thicksim(100,210,120,010).
$$
\end{conjecture}
For Conjecture~\ref{conj:wilf}, the first equivalence has been confirmed in Theorem~\ref{wilf:3}, while the third equivalence was established in~\cite[Thm.~62]{ms}, as 
$$\I_n(110,210,120,010)=\I_n(-,>,\geq)\quad\text{and}\quad\I_n(100,210,120,010)=\I_n(\neq,\geq,\geq).$$
The open problem is the second equivalence. Though unable to prove Conjecture~\ref{conj:wilf}, we find  functional equations that can be applied to compute $|\I_n(011,201)|$  and $|\I_n(\neq,\geq,\geq)|$ for larger  $n$,  in the rest of this section. 

Observe that $011$-avoiding inversion sequences are these whose positive entries are distinct.  For $1\leq m<\ell\leq n$, let us  consider the number $c_{n,m,\ell}$ of $(011,201)$-avoiding inversion sequences of length $n$ with $e_{\ell}=m$ as the unique largest entry.  Let 
$$
c_{n,m}(v)=\sum_{\ell=m+1}^nc_{n,m,\ell}v^{\ell-m-1}\quad\text{and}\quad c_n(u,v)=\sum_{m=0}^{n-1}c_{n,m}(v)u^m.
$$
Define the generating function $C(x;u,v)$ by
$$
C(x;u,v):=\sum_{n\geq1}c_n(u,v)x^n=x+(1+u)x^2+(1+u+uv+2u^2)x^3+\cdots. 
$$
Then $C(x;u,v)$ satisfies the following functional equation. 
\begin{proposition}
For $1\leq m<\ell\leq n$, we have the recursion 
\begin{equation}\label{rec:011,201}
c_{n,m,\ell}=\sum_{j=0}^{m-1}\sum_{i=j+1}^{\ell}c_{n-1,j,i}
\end{equation}
with $c_{n,0,\ell}=\delta_{1,\ell}$.
Equivalently, there holds the functional equation
\begin{equation}\label{func:011,201}
\biggl(1-\frac{ux}{v(1-v)}-\frac{u}{v}\biggr)C(x;u,v)=\frac{x}{1-x}-\frac{ux}{1-v}C(vx;u/v,1)+\frac{u(1-x)}{v}C(x;u,0).
\end{equation}
\end{proposition}
\begin{proof}
Let $\mathcal{C}_{n,m,\ell}$ be the set of sequences $e\in\I_n(011,201)$ with $e_{\ell}=m$ as the unique largest entry. For $e\in\mathcal{C}_{n,m,\ell}$, since $e$ is $201$-avoiding, the entries after $e_{\ell}$ are weakly decreasing. Thus, the deletion of $e_{\ell}$ from $e$ results in an inversion sequence in $\mathcal{C}_{n-1,j,i}$ for some $0\leq j\leq m-1$ and $j+1\leq i\leq \ell$. Conversely, from each inversion sequence  $e\in\mathcal{C}_{n-1,j,i}$ for $0\leq j\leq m-1$ and $j+1\leq i\leq \ell$, we can construct one sequence in $\mathcal{C}_{n,m,\ell}$ by adding a new entry $m$ to the $\ell$-th entry of $e$. This proves the recurrence relation~\eqref{rec:011,201} for $c_{n,m,\ell}$.

For $1\leq m<\ell\leq n$, it follows from~\eqref{rec:011,201} that
$$
c_{n,m,\ell}-c_{n,m-1,\ell}=\sum_{i=m}^{\ell} c_{n-1,m-1,i}. 
$$
Multiplying both sides by $v^{\ell-m-1}$ and summing over $m<\ell\leq n$ gives 
\begin{align*}
&\quad c_{n,m}(v)-v^{-1}(c_{n,m-1}(v)-c_{n,m-1}(0))\\
&=\sum_{\ell=m+1}^nv^{\ell-m-1}\sum_{i=m}^{\ell} c_{n-1,m-1,i}\\
&=-v^{-1}c_{n-1,m-1,m}+\sum_{i=m}^{n} c_{n-1,m-1,i}\sum_{\ell=i}^nv^{\ell-m-1}\\
&=\frac{1}{1-v}(v^{-1}c_{n-1,m-1}(v)-v^{n-m}c_{n-1,m-1}(1))-v^{-1}c_{n-1,m-1}(0)
\end{align*}
for $m\geq1$ with $c_{n,0}(v)=1$.
Multiplying both sides by $u^m$ and summing over $1\leq m<n$ results in 
$$
\biggl(1-\frac{u}{v}\biggr)c_n(u,v)+\frac{u}{v}c_n(u,0)=1+\frac{u}{1-v}(v^{-1}c_{n-1}(u,v)-v^{n-1}c_{n-1}(u/v,1))-\frac{u}{v}c_{n-1}(u,0). 
$$
Multiplying both sides by $x^n$ and summing over $n\geq2$ yields~\eqref{func:011,201} after simplification.
\end{proof}

Next we will use the generating tree technique (see~\cite{bo}) to count $(\neq,\geq,\geq)$-avoiding inversion sequences. For each $e\in\I_n(\neq,\geq,\geq)$, let $m_1(e)$ and $m_2(e)$ be the greatest element and the second greatest element in $\{e_i: i\in[n]\}\cup\{-1\}$, respectively. Introduce the parameters $(p,q)$ of $e$ by 
$$
p=n-m_1(e)\quad\text{and}\quad q=m_1(e)-m_2(e). 
$$
For example, the parameters of $(0,0,2,1)$ is $(2,1)$ and the parameters of $(0,0,0,0)$ is $(4,1)$. We have the following rewriting rule for $(\neq,\geq,\geq)$-avoiding inversion sequences. 

\begin{lemma}\label{lem:sav}
Let $e\in\I_n(\neq,\geq,\geq)$ be an inversion sequence with parameters $(p,q)$. Exactly $p+q$ inversion sequences in $\I_{n+1}(\neq,\geq,\geq)$ when removing their last entries will become $e$, and their parameters are respectively:
  \begin{align*}
 &(p+1,q), (p+1,q-1),\ldots, (p+1,1)\\
&(p,1), (p-1,2),\ldots, (1,p).
 \end{align*}
\end{lemma}
\begin{proof}
It is clear that  the sequence $f:=(e_1,e_2,\ldots,e_n,b)$ is in $\I_{n+1}(\neq,\geq,\geq)$ if and only if $m_2(e)<b\leq n$. We distinguish the following  three cases: 
\begin{itemize}
\item If $m_2(e)<b< m_1(e)$, then $m_1(f)=m_1(e)$ and $m_2(f)=b$. These contribute the parameters $(p+1,q-1), (p+1,q-2),\ldots, (p+1,1)$. 
\item If $b=m_1(e)$, then $m_1(f)=m_1(e)$ and $m_2(f)=m_2(e)$. So this case contributes the parameters $(p+1,q)$.
\item If $m_1(e)<b\leq n$, then $m_1(f)=b$ and $m_2(f)=m_1(e)$. These contribute the parameters $(p,1), (p-1,2),\ldots, (1,p)$.
\end{itemize}
The above three cases together give the rewriting rule for $(\neq,\geq,\geq)$-avoiding inversion sequences. 
\end{proof}

Using the above lemma, we can construct a {\em generating tree} (actually an infinite rooted tree) for $(\neq,\geq,\geq)$-avoiding inversion sequences by representing each element as its parameters as follows: the root is $(1,1)$ and the children of a vertex labelled $(p,q)$ are those generated according to the  rewriting rule in Lemma~\ref{lem:sav}. 
For an inversion sequence $e\in\I_n$, let  $|e|=n$ denote the size of $e$.  Introduce 
$E_{p,q}(x)=\sum_{e}x^{|e|}$, where the sum runs over  all $(\geq,\geq,-)$-inversion sequences with parameters $(p,q)$. Define the formal power series 
$$
E(u,v)=E(x;u,v):=\sum_{p,q\geq1}E_{p,q}(x)u^pv^q=uvx+(uv+u^2v)x^2+\cdots.
$$ We can turn this generating tree into a functional equation as follows. 
\begin{proposition}We have the following functional equation for $E(u,v)$:
\begin{equation}\label{eq:sav}
\biggl(1+\frac{uvx}{1-v}\biggr)E(u,v)=uvx+\biggl(\frac{uvx}{1-v}+\frac{uvx}{u-v}\biggr)E(u,1)-\frac{uvx}{u-v}E(v,1).
\end{equation}
\end{proposition}

\begin{proof}
Since in the generating tree for  $(\neq,\geq,\geq)$-avoiding inversion sequences, each vertex other than the root $(1,1)$ can be generated by a unique parent, we have 
\begin{align*}
E(u,v)&=uvx+x\sum_{p,q\geq1}E_{p,q}(x)\biggl(u^{p+1}\sum_{i=1}^{q}v^i+\sum_{i=1}^{p}u^{p+1-i}v^{i}\biggr)\\
&=uvx+x\sum_{p,q\geq1}E_{p,q}(x)\biggl(\frac{uv(u^p-u^pv^q)}{1-v}+\frac{uv(u^p-v^p)}{u-v}\biggr)\\
&=uvx+\frac{uvx}{1-v}(E(u,1)-E(u,v))+\frac{uvx}{u-v}(E(u,1)-E(v,1)),
\end{align*}
which is equivalent to~\eqref{eq:sav}.
\end{proof}

Can the two functional equations~\eqref{eq:sav} and~\eqref{func:011,201} be solved to prove Conjecture~\ref{conj:wilf}?

\section{Concluding  remarks}
Since the publications of Corteel, Martinez, Savage and Weselcouch~\cite{cor} and Mansour and Shattuck~\cite{ms} on enumeration of inversion sequences avoiding a single pattern of length $3$, patterns in inversion sequences have attracted considerable attentions. Many interesting enumeration results and surprising connections with special integer sequences and functions have already been found~\cite{auli,auli2,bbgr,bgrr,cjl,gue,kl,kl2,lin,lin2,ms,lyan,yan}, some of which are still open. As one of the most attractive conjectures,   Beaton, Bouvel,  Guerrini and Rinaldi~\cite[Conj.~23]{bbgr} suspected that
$$
|\{e\in\I_n(110): \zero(e)=k\}|=|\{\pi\in\S_n(\underline{23}14):\rmin(\pi)=k\}|
$$
 for $1\leq k\leq n$, where $\rmin(\pi)$ denotes the number of right-to-left minima of $\pi$.

In this paper, we finish the  classification of  all the Wilf-equivalences for inversion sequences avoiding pairs of length-$3$ patterns and establish their further connections to some special OEIS sequences and classical combinatorial objects. Can the generating functions, recurrence relations or any general expressions be obtained for the avoidance classes for the pattern pairs that have marked with ``open'' (new in the OEIS~\cite{oeis}) in column 3 in Table~\ref{invseq:pairs2}?

For classical patterns in inversion sequences, one further direction to continue would be to consider a triple (or  multi-tuple) of  length-$3$ patterns or a single pattern of length $4$. In particular, regarding the  integer sequence A279561 in OEIS~\cite{oeis}, the following conjectured enumeration has been discovered by Jun Ma and the second named author.

\begin{conjecture}[Lin and Ma 2019]\label{conj:wilf2} For $n\geq1$, we have 
$$
|\I_n(0012)|=1+\sum_{i=1}^{n-1}{2i\choose i-1}.
$$
In other words, there holds the unbalanced Wilf-equivalence
$$
(0012)\thicksim(021, 120).
$$
\end{conjecture}

\acknowledgements
The authors  thank the anonymous referees for reading carefully the manuscript and providing valuable suggestions. 
The  On-Line Encyclopedia of Integer Sequences~\cite{oeis} created by Neil Sloane is very helpful in this research. 

%

\end{document}